\documentclass[12pt]{article}

\title{Combinatorics of the three-parameter PASEP partition function}
\author{Matthieu Josuat-Verg\`es\thanks{Partially supported by the grant ANR08-JCJC-0011.}\\
\small Universit\'e Paris-sud and LRI, \\
\small 91405 Orsay CEDEX, FRANCE.  \\
\small \texttt{josuat@lri.fr}\\
}

\date{ \small Mathematics Subject Classifications: 05A15, 05A19, 82B23, 60C05.}

\usepackage[english]{babel}
\usepackage[utf8]{inputenc}
\usepackage{amsmath}
\usepackage{amsfonts,amssymb,latexsym,graphics,amsthm}
\usepackage[bookmarks=false]{hyperref}
\usepackage[dvips]{color}
\usepackage[dvips]{graphicx}
\usepackage{pstricks,rotating}
\usepackage{url}

\newcommand{\ket}[1]{\ensuremath{|#1\rangle}}
\newcommand{\bra}[1]{\ensuremath{\langle #1|}}
\newcommand{\braket}[2]{\ensuremath{\langle #1|#2 \rangle}}

\newcommand{\alp}{\tilde\alpha}
\newcommand{\bet}{\tilde\beta}
\newcommand{\qbin}[2]{\genfrac{[}{]}{0pt}{}{#1}{#2}_q}
\newcommand{\tqbin}[2]{\genfrac{[}{]}{0pt}{1}{#1}{#2}_q}
\newcommand{\earrow}{\psset{unit=4.25mm}\begin{pspicture}(0.2,0)(1,0) 
    \psline[linestyle=dotted,dotsep=0.5mm,arrowsize=1.3mm,arrowlength=0.7,arrowinset=0.6]{->}(0.2,0.2)(1,0.2)
   \end{pspicture}}

\newcommand{\wex}{{\rm wex}}
\newcommand{\asc}{{\rm asc}}
\newcommand{\cro}{{\rm cr}}
\newcommand{\ret}{{\rm ret}}
\newcommand{\mot}{{\rm \hbox{31-2}}}
\newcommand{\pk}{{\rm pk}}

\hypersetup{
    unicode=false,          % non-Latin characters in Acrobatâs bookmarks
    pdftoolbar=true,        % show Acrobat toolbar?
    pdfmenubar=true,        % show Acrobat menu?
    pdffitwindow=false,     % window fit to page when opened
    pdfstartview={FitH},    % fits the width of the page to the window
    pdftitle={Combinatorics of the three-parameter PASEP partition function}, 
    pdfauthor={Matthieu Josuat-Vergès (Université Paris XI)},
    pdfsubject={},
    pdfcreator={Creator},   % creator of the document
    pdfproducer={Producer}, % producer of the document
    pdfkeywords={},
    pdfnewwindow=true,      % links in new window
    colorlinks=true,       % false: boxed links; true: colored links
    linkcolor=red,          % color of internal links
    citecolor=blue,        % color of links to bibliography
    filecolor=magenta,      % color of file links
    urlcolor=cyan           % color of external links
}

\newtheorem{thm}{Theorem}[subsection]

\newtheorem{lem}[thm]{Lemma}
\newtheorem{prop}[thm]{Proposition}

\theoremstyle{definition}
\newtheorem{defn}[thm]{Definition}
\newtheorem{rem}[thm]{Remark}

\begin{document}

%\address{LRI and Université Paris-Sud,  Bât. 490, 91405 Orsay, FRANCE}
%\email{josuat@lri.fr}
\maketitle

\begin{abstract}
We consider a partially asymmetric exclusion process (PASEP) on a finite number of sites 
with open and directed boundary conditions. Its partition function was calculated by Blythe, 
Evans, Colaiori, and Essler. It is known to be a generating function of permutation tableaux 
by the combinatorial interpretation of Corteel and Williams.

We prove bijectively two new combinatorial interpretations. The first one is in terms of 
weighted Motzkin paths called Laguerre histories and is obtained by refining a bijection of 
Foata and Zeilberger. Secondly we show that this partition function is the generating function 
of permutations with respect to right-to-left minima, right-to-left maxima, ascents, and 31-2 
patterns, by refining a bijection of Françon and Viennot.

Then we give a new formula for the partition function which generalizes the one of Blythe \& al. 
It is proved in two combinatorial ways. The first proof is an enumeration of lattice paths which 
are known to be a solution of the Matrix Ansatz of Derrida \& al. The second proof relies on a 
previous enumeration of rook placements, which appear in the combinatorial interpretation of a 
related normal ordering problem. We also obtain a closed formula for the moments of 
Al-Salam-Chihara polynomials.
\end{abstract}

%%%%%%%%%%%%%%%%%%%%%%
\section{Introduction}
%%%%%%%%%%%%%%%%%%%%%%

\subsection{The PASEP partition function}

The partially asymmetric simple exclusion process (also called PASEP) is a Markov chain describing
the evolution of particles in $N$ sites arranged in a line, each site being either empty or 
occupied by one particle. Particles may enter the leftmost site at a rate $\alpha\geq0$, go out 
the rightmost site at a rate $\beta\geq0$, hop left at a rate $q\geq0$ and hop right at a rate 
$p>0$ when possible. By rescaling time it is always possible to assume that the latter parameter 
is 1 without loss of generality. It is possible to define either a continuous-time model or a 
discrete-time model, but they are equivalent in the sense that their stationary distributions are 
the same. In this article we only study some combinatorial properties of the partition function. 
For precisions, background about the model, and much more, we refer to 
\cite{BECE, BCEPR, CW1, CW2, DEHP, TS}. 
We refer particularly to the long survey of Blythe and Evans \cite{BE} and all references therein 
to give evidence that this is a widely studied model. Indeed, it is quite rich and some important
features are the various phase transitions, and spontaneous symmetry breaking for example, so that 
it is considered as a fundamental model of nonequilibrium statistical physics.

A method to obtain the stationary distribution and the partition function $Z_N$ of the model is 
the Matrix Ansatz of Derrida, Evans, Hakim and Pasquier \cite{DEHP}. We suppose that $D$ and $E$ 
are linear operators, $\bra{W}$ is a vector, $\ket{V}$ is a linear form, such that:
\begin{equation} \label{ansatz}
DE-qED=D+E, \qquad \bra{W} \alpha E = \bra{W}, \qquad  \beta D\ket{V} = \ket{V},\qquad  \braket WV=1,
\end{equation}
then the non-normalized probability of each state can be obtained by taking the product
$\bra{W} t_1\dots t_N \ket{V}$ where $t_i$ is $D$ if the $i$th site is occupied and $E$ if it is 
empty. It follows that the normalization, or partition function, is given by $\bra{W}(D+E)^N\ket{V}$. 
It is possible to introduce another variable $y$, which is not a parameter of the probabilistic model, 
but is a formal parameter such that the coefficient of $y^{k}$ in the partition function corresponds 
to the states with exactly $k$ particles (physically it could be called a {\it fugacity}). The partition 
function is then:
\begin{equation} \label{ZN}
   Z_N = \bra{W}(yD+E)^N\ket{V},
\end{equation}
which we may take as a definition in the combinatorial point of view of this article 
(see Section~\ref{Zfacts} below for precisions). An interesting property is the symmetry:
\begin{equation} \label{Z_sym}
   Z_N\big(\alpha,\beta,y,q\big) = y^N Z_N\big(\beta,\alpha,\tfrac1y,q\big),
\end{equation}
which can be seen on the physical point of view by exchanging the empty sites with occupied sites.
It can also be obtained from the Matrix Ansatz by using the transposed matrices $D^*$ and $E^*$
and the transposed vectors $\bra{V}$ and $\ket{W}$, which satisfies a similar Matrix Ansatz with
$\alpha$ and $\beta$ exchanged.

\smallskip

In section \ref{paths}, we will use an explicit solution of the Matrix Ansatz \cite{BECE,BCEPR,DEHP},
and it will permit to make use of weighted lattice paths as in \cite{BCEPR}.

\subsection{Combinatorial interpretations}

Corteel and Williams showed in \cite{CW1,CW2} that the stationary distribution of the PASEP
(and consequently, the partition function) has a natural combinatorial interpretation in terms of 
{\it permutation tableaux} \cite{SW}. This can be done by showing that the two operators $D$ and $E$
of the Matrix Ansatz describe a recursive construction of these objects. They have in particular:
\begin{equation} \label{ZTP}
    Z_N = \sum_{T\in PT_{N+1}}  
    \alpha^{-a(T)} \beta^{-b(T)+1} y^{r(T)-1} q^{w(T)},
\end{equation}
where $PT_{N+1}$ is the set of permutation tableaux of size $N+1$, $a(T)$ is the number 
of 1s in the first row, $b(T)$ is the number of unrestricted rows, $r(T)$ is the 
number of rows, and $w(T)$ is the number of superfluous 1s. See Definition \ref{def_PT}
below, and \cite[Theorem~3.1]{CW2} for the original statement. Permutation tableaux are
interesting because of their link with permutations, and it is possible to see $Z_N$ as 
a generating function of permutations. Indeed thanks to the Steingrímsson-Williams 
bijection \cite{SW}, it is also known that \cite{CW2}:
\begin{equation} \label{depart}
    Z_N = \sum_{\sigma\in\mathfrak{S}_{N+1}}  
    \alpha^{ -u(\sigma) } \beta^{ -v(\sigma) } y^{\wex(\sigma)-1} q^{\cro(\sigma)},
\end{equation}
where we use the statistics in the following definition.

\begin{defn} \label{stat_per} Let $\sigma\in\mathfrak{S}_n$. Then:
\begin{itemize}
\item $u(\sigma)$ the number of {\it special} right-to-left minima, {\it i.e.} integers
      $j\in\{1,\dots,n\}$ such that $\sigma(j)=\hbox{min}_{j\leq i\leq n}\sigma(i)$ and
      $\sigma(j)<\sigma(1)$, 
\item $v(\sigma)$ is the number of {\it special} left-to-right maxima, {\it i.e.} integers
      $j\in\{1,\dots,n\}$ such that $\sigma(j)=\hbox{max}_{1\leq i\leq j}\sigma(i)$ and 
      $\sigma(j)>\sigma(1)$,
\item $\wex(\sigma)$ is the number of {\it weak exceedances} of $\sigma$, {\it i.e.} integers
      $j\in\{1,\dots,n\}$ such that $\sigma(j)\geq j$,
\item and $\cro(\sigma)$ is the number of {\it crossings}, {\it i.e.} pairs $(i,j)\in\{1,\dots,n\}^2$ 
      such that either $i<j\leq \sigma(i) <\sigma(j)$ or $\sigma(i)<\sigma(j)<i<j$.
\end{itemize}
\end{defn}
It can already be seen that Stirling numbers and Eulerian numbers appear as special cases of $Z_N$.
We will show that it is possible to follow the statistics in \eqref{depart} through the weighted 
Motzkin paths called {\it Laguerre histories} (see \cite{Co,XGV} and Definition \ref{def_hist} below), 
thanks to the bijection of Foata and 
Zeilberger \cite{Co,FZ,MV}. But we need to study several subtle properties of the bijection to 
follow all four statistics. We obtain a combinatorial interpretation of $Z_N$ in terms of 
{\it Laguerre histories}, see Theorem~\ref{histoires} below. Even more, we will show that the
four statistics in Laguerre histories can be followed through the bijection of Françon and Viennot 
\cite{Co,FV}. Consequently we will obtain in Theorem~\ref{main} below a second new combinatorial
interpretation:
\begin{equation} \label{main_int}
    Z_N = \sum_{\sigma\in\mathfrak{S}_{N+1}}  
    \alpha^{-s(\sigma)+1} \beta^{-t(\sigma)+1} y^{\asc(\sigma)-1}q^{\hbox{\scriptsize \mot}(\sigma)},
\end{equation}
where we use the statistics in the next definition. This was already known in the case $\alpha=1$,
see \cite{Co,CN}.

\begin{defn} \label{stat_per2} Let $\sigma\in\mathfrak{S}_n$. Then:
\begin{itemize}
\item $s(\sigma)$ is the number of right-to-left maxima of $\sigma$ and
      $t(\sigma)$ is the number of right-to-left minima of $\sigma$,
\item $\asc(\sigma)$ is the number of {\it ascents}, {\it i.e.} integers $i$ such that either $i=n$ 
      or $1\leq i \leq n-1$ and $\sigma(i)<\sigma(i+1)$,
\item 31-2$(\sigma)$ is the number of generalized patterns 31-2 in $\sigma$, {\it i.e.} triples of
      integers $(i,i+1,j)$ such that $1\leq i<i+1<j\leq n$ and $\sigma(i+1)<\sigma(j)<\sigma(i)$.
%\item 31-2$(\sigma,j)$ is the number of patterns 31-2 such that $j$ correspond to the 2,
%      {\it i.e.} integers $i$ such that $1<i+1<j$ and $\sigma(i+1)<\sigma(j)<\sigma(i)$.
\end{itemize}
\end{defn}

\subsection{Exact formula for the partition function}

An exact formula for $Z_N$ was given by Blythe, Evans, Colaiori, Essler \cite[Equation (57)]{BECE} 
in the case where $y=1$. It was obtained from the eigenvalues and eigenvectors of the operator 
$D+E$ as defined in \eqref{defD} and \eqref{defE} below. This method gives an integral form
for $Z_N$, which can be simplified so as to obtain a finite sum rather than an integral. Moreover 
this expression for $Z_N$ was used to obtain various properties of the large system size limit, 
such as phases diagrams and currents. Here we generalize this result since we also have the 
variable $y$, and the proofs are combinatorial. This is an important result since it is generally
accepted that most interesting properties of a model can be derived from the partition function.

\begin{thm}\label{Z_th}
Let $\alp=(1-q)\frac1\alpha-1$ and $\bet=(1-q)\frac1\beta-1$. We have:
\begin{equation} \label{Z}
    Z_N = \frac 1{(1-q)^N} \sum_{n=0}^N R_{N,n}(y,q) B_n( \alp , \bet , y, q),
\end{equation}
where
\begin{equation}
    R_{N,n}(y,q) = \sum_{i=0}^{\lfloor \frac{N-n}2 \rfloor } (-y)^i
%    \left( \tbinom{2N}{N-n-2k} - \tbinom{2N}{N-n-2k-2} \right)
    q^{\binom{i+1}2} \tqbin{n+i}i
    \sum_{j=0}^{N-n-2i}y^j\left( \tbinom Nj \tbinom N{n+2i+j}-
     \tbinom N{j-1} \tbinom N{n+2i+j+1}  \right)
\end{equation}
and
\begin{equation}
    B_n(\alp,\bet,y,q) = \sum_{k=0}^n \qbin nk \alp^k (y\bet)^{n-k}.
\end{equation}
\end{thm}
In the case where $y=1$, one sum can be simplified by the Vandermonde identity 
$\sum_j \binom Nj \binom N{m-j}=\binom{2N}m$, and we recover the expression 
given in \cite[Equation (54)]{BECE} by Blythe \& al:
\begin{equation}
    R_{N,n}(1,q) = \sum_{i=0}^{\lfloor \frac{N-n}2 \rfloor } (-1)^i
    \left( \tbinom{2N}{N-n-2i} - \tbinom{2N}{N-n-2i-2} \right)
    q^{\binom{i+1}2} \tqbin{n+i}i.
%    \sum_{j=0}^{N-n-2i}y^j\left( \tbinom Nj \tbinom N{n+2i+j}-
%     \tbinom N{j-1} \tbinom N{n+2i+j+1}  \right)
\end{equation}
In the case where $\alpha=\beta=1$, it was known \cite{CJPR,MJV} that $(1-q)^{N+1}Z_N$ is equal to:
\begin{equation} \label{cas1}
%  Z_{N} =  \tfrac 1{(1-q)^{N+1}}
  \sum\limits_{k=0}^{N+1} (-1)^k
    \Bigg( \sum\limits_{j=0}^{N+1-k} y^j\Big( \tbinom{N+1}{j}\tbinom{N+1}{j+k} -
    \tbinom{N+1}{j-1}\tbinom{N+1}{j+k+1}\Big) \Bigg) 
  \left( \sum\limits_{i=0}^k y^iq^{i(k+1-i)} \right)
\end{equation}
(see Remarks \ref{comp1} and \ref{comp2} for a comparison between this previous result and the 
new one in Theorem~\ref{Z_th}). And in the case where $y=q=1$, from a recursive construction of
permutation tableaux \cite{CN} or lattice paths combinatorics \cite{BCEPR} it is known that:
\begin{equation}
  Z_{N} =  \prod_{i=0}^{N-1} \left(\frac1\alpha + \frac1\beta + i\right).
\end{equation}

The first proof of \eqref{Z} is a purely combinatorial enumeration of some weighted Motzkin
paths defined below in \eqref{Z_w}, appearing from explicit representations of the operators
$D$ and $E$ of the Matrix Ansatz. It partially relies on results of \cite{CJPR,MJV} through 
Proposition~\ref{RNn} below. In contrast, the second proof does not use a particular
representation of the operators $D$ and $E$, but only on the combinatorics of the normal ordering
process. It also relies on previous results of \cite{MJV} (through Proposition~\ref{hats} below),
but we will sketch a self-contained proof.

This article is organized as follows. In Section~\ref{Zfacts} we recall known facts about the
PASEP partition function $Z_N$, mainly to explain the Matrix Ansatz. In Section~\ref{bij} we
prove the two new combinatorial interpretations of $Z_N$, starting from \eqref{depart} 
and using various properties of bijections of Foata and Zeilberger, Françon and Viennot. 
Sections \ref{paths} and \ref{rooks} respectively contain the the two proofs of the exact 
formula for $Z_N$ in Equation \eqref{Z}. In Section~\ref{ALSC} we show that the first proof
of the exact formula for $Z_N$ can be adapted to give a formula for the moments of 
Al-Salam-Chihara polynomials. Finally in Section~\ref{num} we review the numerous classical
integer sequences which appear as specializations or limit cases of $Z_N$.

%%%%%%%%%%%%%%%%%%%%%%%%%%
\section*{Acknowledgement}
%%%%%%%%%%%%%%%%%%%%%%%%%%

I thank my advisor Sylvie Corteel for her advice, support, help and kindness.
I thank Einar Steingrímsson, Lauren Williams and Jiang Zeng for their help.

%%%%%%%%%%%%%%%%%%%%%%%%%%%%%%%%%%%%%%%%%%%%%%%%%%%%%%%%%%%%%%%
\section{Some known properties of the partition function $Z_N$}
%%%%%%%%%%%%%%%%%%%%%%%%%%%%%%%%%%%%%%%%%%%%%%%%%%%%%%%%%%%%%%%
\label{Zfacts}

As said in the introduction, the partition function $Z_N$ can be derived by taking the 
product $\bra{W}(yD+E)^N\ket{V}$ provided the relations \eqref{ansatz} are satisfied.
It may seem non-obvious that $\bra{W}(yD+E)^N\ket{V}$ does not depend on a particular 
choice of the operators $D$ and $E$, and the existence of such operators $D$ and $E$ 
is not clear.

\smallskip

The fact that $\bra{W}(yD+E)^N\ket{V}$ is well-defined without making $D$ and $E$ explicit,
in a consequence of the existence of normal forms. More precisely, via the commutation 
relation $DE-qED=D+E$ we can derive polynomials $c^{(N)}_{i,j}$ in $y$ and $q$ with 
non-negative integer coefficients such that we have the {\it normal form}:
\begin{equation}
   (yD+E)^N = \sum_{i,j\geq 0} c^{(N)}_{i,j} E^i D^j
\end{equation}
(this is a finite sum). See \cite{BHPSD} for other combinatorial interpretation of normal 
ordering problems. It turns out that the $c^{(N)}_{i,j}$ are uniquely defined if we require
the previous equality to hold for any value of $\alpha$, $\beta$, $y$ and $q$, considered as
indeterminates. Then the partition function is:
\begin{equation}
   Z_N(\alpha,\beta,y,q) = \bra{W}(yD+E)^N\ket{V} = 
   \sum_{i,j\geq 0} c^{(N)}_{i,j} \alpha^{-i} \beta^{-j}.
\end{equation}
Indeed, this expression is valid for any choice of $\bra{W}$, $\ket{V}$, $D$ and $E$ since we 
only used the relations \eqref{ansatz} to obtain it. In particular $Z_N$ is a polynomial in $y$, 
$q$, $\frac{1}{\alpha}$ and $\frac{1}{\beta}$ with non-negative coefficients. For convenience 
we also define:
\begin{equation}
 \bar Z_N\big(\alpha,\beta,y,q\big)= Z_N\big(\tfrac{1}{\alpha},\tfrac{1}{\beta},y,q\big). 
\end{equation}
For example the first values are:
\[
 \bar Z_0=1, \qquad  \bar  Z_1=\alpha+y\beta, \qquad
 \bar Z_2= \alpha^2 + y(\alpha + \beta + \alpha\beta + \alpha\beta q) + y^2\beta^2,
\]
\begin{eqnarray*}
\bar Z_3 &=& y^3\beta^3 + \left( \alpha\beta^2q+\alpha\beta^2 + \alpha + \alpha\beta + 
       \alpha\beta^2q ^2+\beta+\beta^2 q 
       + 2\,a\beta q + 2\beta^{2} \right) y^2   \\ 
& & + \left( 2\alpha^2+\alpha^2q+\alpha+\beta\alpha^2q^{2} + \beta\alpha^2 + \beta\alpha^2q +
      \alpha \beta + \beta + 2 \alpha \beta q \right) y+\alpha^3.
\end{eqnarray*}

Even if it is not needed to compute the first values of $Z_N$, it is useful to have
explicit matrices $D$ and $E$ satisfying \eqref{ansatz}. The best we could hope is 
finite-dimensional matrices with non-negative entries, however this is known to be 
incompatible with the existence of phase transitions in the model (see section 2.3.3 
in \cite{BE}). Let $\alp=(1-q)\frac{1}{\alpha}-1$ and $\bet=(1-q)\frac{1}{\beta}-1$, 
a solution of the Matrix Ansatz \eqref{ansatz} is given by the following matrices
$D=(D_{i,j})_{i,j\in\mathbb{N}}$ and $E=(E_{i,j})_{i,j\in\mathbb{N}}$ (see \cite{DEHP}):
\begin{equation} \label{defD}
(1-q)D_{i,i}= 1 + \bet q^i, \qquad (1-q)D_{i,i+1}= 1-\alp\bet q^i,
\end{equation}
\begin{equation} \label{defE}
(1-q)E_{i,i}= 1 + \alp q^i, \qquad (1-q)E_{i+1,i}= 1-  q^{i+1},
\end{equation}
all other coefficients being 0, and vectors:
\begin{equation}
\bra{W}=(1,0,0,\dots), \qquad \ket{V} = (1,0,0,\dots)^*,
\end{equation}
({\it i.e.} $\ket{V}$ is the transpose of $\bra{W}$). Even if infinite-dimensional, they
have the nice property of being tridiagonal and this lead to a combinatorial interpretation
of $Z_N$ in terms of lattice paths \cite{BCEPR}. Indeed, we can see $yD+E$ as a transfer matrix 
for walks in the non-negative integers, and obtain that $(1-q)^N Z_N$ is the sum of weights of 
Motzkin paths of length $N$ with weights:
\begin{equation} \label{Z_w}
\parbox{14.5cm}{
\begin{itemize}
\item $1-q^{h+1}$ for a step $\nearrow$ starting at height $h$,
\item $(1+y)+(\alp+y\bet)q^h$ for a step $\rightarrow$ starting at height $h$,
\item $y(1-\alp\bet q^{h-1})$ for a step $\searrow$ starting at height $h$.
\end{itemize}
}
\end{equation}
We recall that a Motzkin path is similar to a Dyck path except that there may be horizontal
steps, see Figures~\ref{ex_FZ}, \ref{fig}, \ref{ex_FV}, \ref{decomp_ex} further. These weighted
Motzkin paths are our starting point to prove Theorem~\ref{Z_th} in Section~\ref{paths}.

We have sketched how the Motzkin paths appear as a combinatorial interpretation of $Z_N$ starting
from the Matrix Ansatz. However it is also possible to obtain a direct link between the PASEP and
the lattice paths, independently of the results of Derrida \& al. This was done by Brak \& al 
in \cite{BCEPR}, in the even more general context of the PASEP with five parameters.

%%%%%%%%%%%%%%%%%%%%%%%%%%%%%%%%%%%%%%%%%%%%%%%%%%%%%
\section{Combinatorial interpretations of $Z_N$}
%%%%%%%%%%%%%%%%%%%%%%%%%%%%%%%%%%%%%%%%%%%%%%%%%%%%%
\label{bij}

In this section we prove the two new combinatorial interpretation of $Z_N$. Firstly we prove
the one in terms of Laguerre histories (Theorem~\ref{histoires} below), by means of a bijection
originally given by Foata and Zeilberger. Secondly we prove the one in terms in permutations 
(Theorem~\ref{main} below).

\subsection{Permutation tableaux and Laguerre histories}
We recall here the definition of permutation tableaux and their statistics needed to state 
the previously known combinatorial interpretation \eqref{ZTP}.

\begin{defn}[\cite{SW}] \label{def_PT}
Let $\lambda$ be a Young diagram (in English notation), possibly with empty rows but with no 
empty column. A complete filling of $\lambda$ with 0's and 1's is a {\it permutation tableau} if:
\begin{itemize}
\item for any cell containing a 0, all cells above in the same column contain a 0, or all cells 
      to the left in the same row contain a 0,
\item there is at least a 1 in each column.
\end{itemize}
A cell containing a 0 is {\it restricted} if there is a 1 above. A row is {\it restricted} if it 
contains a restricted 0, and {\it unrestricted} otherwise. A cell containing a 1 is {\it essential} 
if it is the topmost 1 of its column, otherwise it is {\it superfluous}. The {\it size} of such a
permutation tableaux is the number of rows of $\lambda$ plus its number of columns.
\end{defn}

To prove our new combinatorial interpretations, we will give bijections linking the previously-known 
combinatorial interpretation \eqref{depart}, and the new ones. The main combinatorial object
we use are the Laguerre histories, defined below.

\begin{defn}[\cite{XGV}] \label{def_hist}
A {\it Laguerre history} of size $n$ is a weighted Motzkin path of $n$ 
steps such that:
\begin{itemize}
\item the weight of a step $\nearrow$ starting at height $h$ is $yq^i$ for some
  $i\in\left\{0,\dots,h\right\}$,
\item the weight of a step $\rightarrow$ starting at height $h$ is either $yq^i$ for
  some $i\in\left\{0,\dots,h\right\}$ or $q^i$ for some $i\in\left\{0,\dots,h-1\right\}$,
\item the weight of a step $\searrow$ starting at height $h$ is $q^i$ for some
  $i\in\left\{0,\dots,h-1\right\}$.
\end{itemize}
The {\it total weight} of the Laguerre history is the product of the weights of its steps. 
We call a {\it type 1 step}, any step having weight $yq^h$ where $h$ is its starting height.
We call a {\it type 2 step}, any step having weight $q^{h-1}$ where $h$ is its starting height.
\end{defn}

\smallskip

As shown by
P. Flajolet \cite{Fla}, the weighted Motzkin paths appear in various combinatorial contexts
in connexion with some continued fractions called J-fractions. We also recall an important 
fact from combinatorial theory of orthogonal polynomials.

\begin{prop}[Flajolet \cite{Fla}, Viennot \cite{XGV}] \label{mu_ortho}
If an orthogonal sequence $\{P_n\}_{n\in\mathbb{N}}$ is defined by the three-term 
recurrence relation
\begin{equation}
   xP_n(x) = P_{n+1}(x) + b_n P_n(x) + \lambda_n P_{n-1}(x),
\end{equation}
then the moment generating function has the J-fraction representation
\begin{equation}
  \sum_{n=0}^\infty \mu_n t^n = \cfrac{1}{1 - b_0t -  
        \cfrac{\lambda_1 t^2}{1 - b_1t -  \cfrac{\lambda_2t^2}{\ddots
}}},
\end{equation}
equivalently the $n$th moment $\mu_n$ is the sum of weights of Motzkin paths of length $n$ 
where the weight of a step $\nearrow$ (respectively $\rightarrow$, $\searrow$) starting at 
height $h$ is $a_h$ (respectively $b_h$, $c_h$) provided $\lambda_n=a_{n-1}c_n$.
\end{prop}

\begin{rem} \label{rem_poly}
The sum of weights of Laguerre histories of length $n$ is the $n$th moment of some $q$-Laguerre
polynomials (see \cite{KSZ08}), which are a special case of rescaled Al-Salam-Chihara polynomials.
On the other hand $Z_N$ is the $N$th moment of shifted Al-Salam-Chihara polynomials (see Section~\ref{ALSC}).
We will use the Laguerre histories to derive properties of $Z_N$, however they are
related with two different orthogonal sequences.
\end{rem}

\smallskip

\subsection{The Foata-Zeilberger bijection} Foata and Zeilberger gave a bijection between 
permutations and Laguerre histories in \cite{FZ}. It has been extended by de Médicis and Viennot
\cite{MV}, and Corteel \cite{Co}. In particular, Corteel showed that through this bijection
$\Psi_{FZ}$ we can follow the number weak exceedances and crossings \cite{Co}. The bijection
$\Psi_{FZ}$ links permutations in $\mathfrak{S}_n$ and Laguerre histories of $n$ steps. The $i$th 
step of $\Psi_{FZ}(\sigma)$ is:
\begin{itemize}
\item a step $\nearrow$ if $i$ is a {\it cycle valley}, {\it i.e.} $\sigma^{-1}(i) > i < \sigma(i)$,
\item a step $\searrow$ if $i$ is a {\it cycle peak}, {\it i.e.} $\sigma^{-1}(i) < i > \sigma(i)$,
\item a step $\rightarrow$ in all other cases.
\end{itemize}
And the weight of the $i$th step in $\Psi_{FZ}(\sigma)$ is $y^\delta q^j$ with:
\begin{itemize}
\item $\delta=1$ if $i\leq\sigma(i)$ and 0 otherwise,
\item $j=\#\{\; k \mid k < i \leq \sigma(k) < \sigma(i) \;\} $ 
       if $i \leq \sigma(i)$,
\item $j=\#\{\; k \mid \sigma(i) < \sigma(k) < i < k \; \} $ 
       if $\sigma(i) < i$.
\end{itemize}
It follows that the total weight of $\Psi_{FZ}(\sigma)$ is $y^{\wex(\sigma)} q^{\cro(\sigma)}$.
To see the statistics $\wex$ and $\cro$ in a permutation $\sigma$, it is practical to represent
$\sigma$ by an arrow diagram. We draw $n$ points in a line, and draw an arrow from the $i$th 
point to the $\sigma(i)$th point for any $i$. This arrow is above the axis if $i\leq\sigma(i)$, 
below the axis otherwise. Then $\wex(\sigma)$ is the number of arrows above the axis, and 
$\cro(\sigma)$ is the number of proper intersection between arrows plus the number of chained
arrows going to the right. See Figure~\ref{ex_FZ} for an example with  $\sigma=672581493$, so that
$\wex(\sigma)=5$ and $\cro(\sigma)=7$.

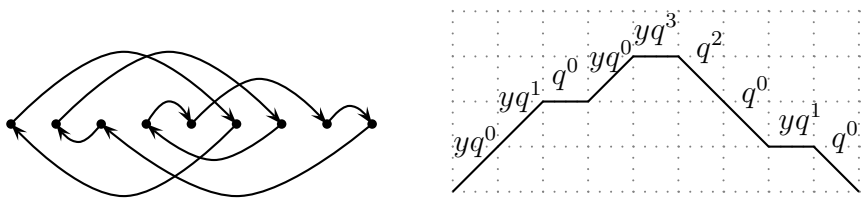
\begin{figure}[h!tp] \psset{unit=6mm} \centering
\begin{pspicture}(1,-2)(7,1)
\psdots(1,0)(2,0)(3,0)(4,0)(5,0)(6,0)(7,0)(8,0)(9,0)
\pscurve[arrowscale=1.5]{->}(1,0)(3.5,1.6)(6,0)     
\pscurve[arrowscale=1.5]{->}(2,0)(4.5,1.6)(7,0)
\pscurve[arrowscale=1.5]{->}(3,0)(2.5,-0.4)(2,0) 
\pscurve[arrowscale=1.5]{->}(4,0)(4.5,0.5)(5,0)
\pscurve[arrowscale=1.5]{->}(5,0)(6.5,1)(8,0)
\pscurve[arrowscale=1.5]{->}(6,0)(3.5,-1.6)(1,0)
\pscurve[arrowscale=1.5]{->}(7,0)(5.5,-0.8)(4,0)
\pscurve[arrowscale=1.5]{->}(8,0)(8.5,0.4)(9,0)
\pscurve[arrowscale=1.5]{->}(9,0)(6,-1.6)(3,0)
\end{pspicture}
\hspace{2cm} \psset{unit=6mm}
\begin{pspicture}(0,-0.5)(9,4)
\psgrid[gridcolor=gray,griddots=4,subgriddiv=0,gridlabels=0](0,0)(9,4)
\psline(0,0)(1,1)(2,2)(3,2)(4,3)(5,3)(6,2)(7,1)(8,1)(9,0)
\rput(0.5,1.1){$yq^0$}
\rput(1.5,2.1){$yq^1$}
\rput(2.5,2.6){$q^0$}
\rput(3.5,3){$yq^0$}
\rput(4.5,3.6){$yq^3$}
\rput(5.7,3.2){$q^2$}
\rput(6.7,2.2){$q^0$}
\rput(7.7,1.6){$yq^1$}
\rput(8.7,1.2){$q^0$}
\end{pspicture}
\caption{\label{ex_FZ}  The permutation $\sigma=672581493$ and its image $\Psi_{FZ}(\sigma)$. }
\end{figure}

\begin{lem} \label{FZ_lrmax}
  Let $\sigma\in\mathfrak{S}_n$, and $1\leq i\leq n$. Then $i$ is a left-to-right 
  maximum of $\sigma$ if and only if the $i$th step of $\Psi_{FZ}(\sigma)$ is a type 1 step
  (as in Definition \ref{def_hist}).
\end{lem}

\begin{proof}
Let us call a $(\sigma,i)$-sequence a strictly increasing maximal sequence of integers 
$u_1,\dots,u_j$ such that $\sigma(u_k)=u_{k+1}$ for any $1\leq k\leq j-1$, and also such that
$u_1 < i < u_j$. By maximality of the sequence, $u_1$ is a cycle valley and $u_j$ is a cycle 
peak. The number of such sequences is the difference between the number of cycle valleys and 
cycle peaks among $\{1,\dots,i-1\}$, so it is the starting height $h$ of the $i$th step in 
$\Psi_{FZ}(\sigma)$.

\smallskip

Any left-to-right maximum is a weak exceedance, so $i$ is a left-to-right maxima of 
$\sigma$ if and only if  $i\leq\sigma(i)$ and there exists no $j$ such that 
$j < i \leq \sigma(i) < \sigma(j)$. This is also equivalent to the fact that 
$i \leq \sigma(i)$, and there exists no two consecutive elements $u_k$, $u_{k+1}$ of a 
$(\sigma,i)$-sequence such that $u_k < i \leq \sigma(i) < u_{k+1}$. This is also equivalent to 
the fact that $i\leq \sigma(i)$, and any $(\sigma,i)$-sequence contains two consecutive 
elements $u_k$, $u_{k+1}$ such that $u_k < i \leq u_{k+1} < \sigma(i)$.

By definition of the bijection $\Psi_{FZ}$ it is equivalent to the fact that the $i$th step 
of $\Psi_{FZ}(\sigma)$ has weight $yq^h$, {\it i.e.} the $i$th step is a type 1 step.
\end{proof}

\begin{lem} \label{FZ_rlmin}
  Let $\sigma\in\mathfrak{S}_n$, and $1\leq i\leq n$. We suppose $i\neq\sigma(i)$.
  Then $i$ is a right-to-left minima of $\sigma$ 
  if and only if the $i$th step of $\Psi_{FZ}(\sigma)$ is a type 2 step.
\end{lem}

\begin{proof}
We have to pay attention to the fact that a right-to-left minimum can be a fixed point and we
only characterize the non-fixed points here. This excepted, the proof is similar to the one 
of the previous lemma.
\end{proof}

%Note that a special right-to-left minimum is not a fixed point. Indeed, if $i$ is 
%a right-to-left minima and a fixed point of $\sigma$, then $\sigma$ stabilizes the
%interval $\{1,\dots,i-1\}$ and it is not possible that $\sigma(1)>\sigma(i)$. 

Before we can use the bijection $\Psi_{FZ}$ we need a slight modification of the known
combinatorial interpretation \eqref{depart}, given in the following lemma.

\begin{lem} \label{vprim}
We have:
\begin{equation} \label{depart2}
    \bar Z_N = \sum_{\sigma\in\mathfrak{S}_{N+1}}  
    \alpha^{ u'(\sigma) } \beta^{ v(\sigma) } y^{\wex(\sigma)-1} q^{\cro(\sigma)},
\end{equation}
where $u'(\sigma)$ is the number of right-to-left minima $i$ of $\sigma$ satisfying 
$\sigma^{-1}(N+1) < i$.
\end{lem}

\begin{proof}
This just means that in \eqref{depart} we can replace the statistic $u$ with $u'$, and
this can be done via a simple bijection. For any $\sigma\in\mathfrak{S}_{N+1}$, let 
$\tilde\sigma$ be the reverse complement of $\sigma^{-1}$, {\it i.e.} $\sigma(i)=j$ 
if and only if $\tilde\sigma(N+2-j)=N+2-i$. It is routine to check that 
$u(\sigma)=u'(\tilde\sigma)$, $\wex(\sigma)=\wex(\tilde\sigma)$, and $v(\sigma)=v(\tilde\sigma)$. 
Moreover, one can check that the arrow diagram of $\tilde\sigma$ is obtained from the one of 
$\sigma$ by a vertical symmetry and arrow reversal, so that $\cro(\sigma)=\cro(\tilde\sigma)$. 
So \eqref{depart} and the bijection $\sigma\mapsto\tilde\sigma$ prove \eqref{depart2}.
\end{proof}

From Lemmas~\ref{FZ_lrmax}, \ref{FZ_rlmin}, and \ref{vprim} it possible to give a combinatorial
interpretation of $\bar Z_N$ in terms of the Laguerre histories. We start from the statistics in
$\mathfrak{S}_{N+1}$ described in Definition~\ref{stat_per}, then from \eqref{depart2} and the
properties of $\Psi_{FZ}$ we obtain the following theorem.

\begin{thm} \label{histoires}
The polynomial $y\bar Z_N$ is the generating function of Laguerre histories of $N+1$ steps, 
where: 
\begin{itemize}
\item the parameters $y$ and $q$ are given by the total weight of the path,
\item $\beta$ counts the type 1 steps, except the first one, 
\item $\alpha$ counts the type 2 steps which are to the right of any type 1 step.
\end{itemize}
\end{thm}

\begin{proof}
Let $\sigma\in\mathfrak{S}_{N+1}$.
The smallest left-to-right maximum of $\sigma$ is 1, and any other left-to-right maximum $i$ is 
such that $\sigma(1)<\sigma(i)$. So $1$ is the only left-to-right maximum which is not special.
So by Lemma~\ref{FZ_lrmax}, $v(\sigma)$ is the number of type 1 steps in $\Psi_{FZ}(\sigma)$,
minus 1.

Moreover, $\sigma^{-1}(N+1)$ is the largest left-to-right maximum of $\sigma$. Let $i$ be
a right-to-left minimum of $\sigma$ such that $\sigma^{-1}(N+1)<i$. We have $i\neq\sigma(i)$,
otherwise $\sigma$ would stabilize the interval $\{i+1,\dots,N+1\}$ and this would contradict
$\sigma^{-1}(N+1)<i$. So we can apply Lemma~\ref{FZ_rlmin}, and it comes that $u'(\sigma)$
is the number of type 2 steps in $\Psi_{FZ}(\sigma)$, which are to the right of any type 1 step.
So \eqref{depart2} and the bijection $\Psi_{FZ}$ prove the theorem.
\end{proof}

Before ending this subsection, we sketch how to recover a known result in the case $q=0$
from Theorem~\ref{histoires}. This was given in Section~3.2 of \cite{BGR} (see also Section
3.6 in \cite{BE}) and proved via generating functions. For any Dyck path $D$, let $\ret(D)$ 
be the number of returns to height 0, for example $\ret(\nearrow\searrow)=1$ and
$\ret(\nearrow\searrow\nearrow\searrow)=2$, and the empty path $\cdot$ satisfies $\ret(\cdot)=0$. 
The result is the following.

\begin{prop}[Brak, de Gier, Rittenberg] When $y=1$ and $q=0$, the partition function is 
$Z_N=\sum (\frac1\beta)^{\ret(D_1)}(\frac1\alpha)^{\ret(D_2)}$ where the sum is over
pairs of Dyck paths $(D_1,D_2)$ whose lengths sum to $2N$.
\end{prop}

\begin{proof}
When $q=0$ we can remove any step with weight 0 in the Laguerre histories. When $y=1$, 
to distinguish the two kinds of horizontal steps we introduce another kind of paths. Let us call 
a {\it bicolor} Motzkin path, a Motzkin path with two kinds of horizontal steps $\earrow$ and 
$\rightarrow$, and such that there is no $\earrow$ at height 0. From Theorem~\ref{histoires},
if $y=1$ and $q=0$ then $\beta\bar Z_N$ is the generating function of bicolor Motzkin paths 
$M$ of length $N+1$, where:
\begin{itemize}
\item there is a weight $\beta$ on each step $\nearrow$ or $\rightarrow$ starting at height 0,
\item there is a weight $\alpha$ on each step $\searrow$ or $\earrow$ starting at height 1 and being
 to the right of any step with a weight $\beta$.
\end{itemize}
There is a bijection between these bicolor Motzkin paths, and Dyck paths of length $2N+2$ (see de 
Médicis and Viennot \cite{MV}). To obtain the Dyck path $D$, each step $\nearrow$ in the bicolor Motzkin 
path $M$ is replaced with a sequence of two steps $\nearrow\nearrow$. Similarly, each step $\rightarrow$ 
is replaced with $\nearrow\searrow$, each step $\earrow$ is replaced with $\searrow\nearrow$, each step
$\searrow$ is replaced with $\searrow\searrow$. When some step $s\in\{\nearrow,\rightarrow,\searrow\}$
in $M$ has a weight $\beta$ or $\alpha$, and is transformed into steps 
$(s_1,s_2)\in\{\nearrow,\rightarrow,\searrow\}^2$ in $D$, we choose to put the weight $\beta$ or $\alpha$
on $s_1$. It appears that $D$ is a Dyck path of length $2N+2$ such that:
\begin{itemize}
\item there is a weight $\beta$ on each step $\nearrow$ starting at height 0,
\item there is a weight $\alpha$ on each step $\searrow$ starting at height 2 and being to the right 
      of any step with weight $\beta$.
\end{itemize}
Then $D$ can be factorized into $D_1\nearrow D_2\searrow$ where $D_1$ and $D_2$ are Dyck paths whose 
lengths sum to $2N$, and up to a factor $\beta$ it can be seen that $\beta$ (respectively $\alpha$) 
counts the returns to height 0 in $D_1$ (respectively $D_2$). More precisely the $\beta$s are on the 
steps $\nearrow$ starting at height 0 but there are as many of them as the number of returns to height 0.
See Figure~\ref{dycks} for a an example.
\end{proof}

\begin{figure}[h!tp] \psset{unit=3mm} \centering
\begin{pspicture}(-1.6,0)(10,4)
\psgrid[gridcolor=gray,griddots=4,subgriddiv=0,gridlabels=0](0,0)(10,4)
\psline(0,0)(1,1)
\psline[linestyle=dotted,dotsep=0.5mm,linewidth=0.5mm](1,1)(2,1)
\psline(2,1)(3,0)(4,0)(5,1)(6,1)(7,2)(8,1)
\psline[linestyle=dotted,dotsep=0.5mm,linewidth=0.5mm](8,1)(9,1)
\psline(9,1)(10,0)
\rput(-1.7,1.6){$M=$}
\rput(0.5,1.5){$\beta$}\rput(3.5,0.5){$\beta$}\rput(4.5,1.5){$\beta$}
\rput(8.5,1.5){$\alpha$}\rput(9.5,1.5){$\alpha$}
\end{pspicture}
\hspace{1.4cm} \psset{unit=3mm}
\begin{pspicture}(20,5)
\psgrid[gridcolor=gray,griddots=4,subgriddiv=0,gridlabels=0](0,0)(20,5)
\psline(0,0)(1,1)(2,2)(3,1)(4,2)(5,1)(6,0)(7,1)(8,0)(9,1)(10,2)(11,3)(12,2)(13,3)(14,4)(15,3)(16,2)(17,1)(18,2)(19,1)(20,0)
\rput(-1.7,1.6){$D=$}
\rput(0.5,1.5){$\beta$}\rput(6.5,1.5){$\beta$}\rput(8.5,1.5){$\beta$}
\rput(16.5,2.5){$\alpha$}\rput(18.5,2.5){$\alpha$}
\end{pspicture}
\\[3mm]
\begin{pspicture}(8,5)
\psgrid[gridcolor=gray,griddots=4,subgriddiv=0,gridlabels=0](0,0)(8,4)
\psline(0,0)(1,1)(2,2)(3,1)(4,2)(5,1)(6,0)(7,1)(8,0)
\rput(-1.8,1.6){$D_1=$}
\rput(0.5,1.5){$\beta$}\rput(6.5,1.5){$\beta$}
\end{pspicture}
\hspace{2cm}
\begin{pspicture}(10,5)
\psgrid[gridcolor=gray,griddots=4,subgriddiv=0,gridlabels=0](0,0)(10,4)
\psline(0,0)(1,1)(2,2)(3,1)(4,2)(5,3)(6,2)(7,1)(8,0)(9,1)(10,0)
\rput(-1.8,1.6){$D_2=$}
\rput(7.5,1.5){$\alpha$}\rput(9.5,1.5){$\alpha$}
\end{pspicture}

\caption{\label{dycks}
The bijection between $M$, $D$ and $(D_1,D_2)$.}
\end{figure}
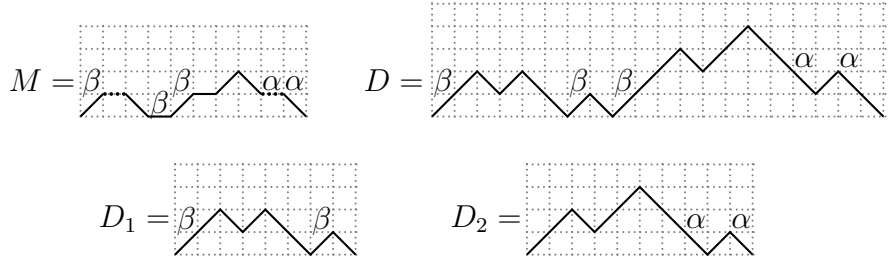

\subsection{The Françon-Viennot bijection} This bijection was given in \cite{FV}. We use here
the definition of this bijection given in \cite{Co}. The map $\Psi_{FV}$ is a bijection between
permutations of size $n$ and Laguerre histories of $n$ steps. Let $\sigma\in\mathfrak{S}_n$, 
$j\in\{1,\dots,n\}$ and $k=\sigma(j)$. Then the $k$th step of $\Psi_{FV}(\sigma)$ is:
\begin{itemize}
\item a step $\nearrow$ if $k$ is a {\it valley}, {\it i.e.} $\sigma(j-1)>\sigma(j)<\sigma(j+1)$,
\item a step $\searrow$ if $k$ is a {\it peak}, {\it i.e.} $\sigma(j-1)<\sigma(j)>\sigma(j+1)$,
\item a step $\rightarrow$ if $k$ is a {\it double ascent}, {\it i.e.}
$\sigma(j-1)<\sigma(j)<\sigma(j+1)$, or a {\it double descent}, {\it i.e.}
$\sigma(j-1)>\sigma(j)>\sigma(j+1)$.
\end{itemize}
This is done with the convention that $\sigma(n+1)=n+1$, in particular $n$ is always an ascent 
of $\sigma\in\mathfrak{S}_n$. Moreover the weight of the $k$th step is $y^\delta q^i$ where 
$\delta=1$ if $j$ is an ascent and $0$ otherwise, and $i=$ 31-2$(\sigma,j)$. This number
31-2$(\sigma,j)$ is the number of patterns 31-2 such that $j$ correspond to the 2,
{\it i.e.} integers $i$ such that $1<i+1<j$ and $\sigma(i+1)<\sigma(j)<\sigma(i)$.
A consequence of the definition is that the total weight of $\Psi_{FV}(\sigma)$ 
is $y^{\asc(\sigma)}q^{\hbox{\scriptsize \mot}(\sigma)}$. See Figures~\ref{fig} and \ref{ex_FV}
for examples.

\begin{figure}[h!tp]\psset{unit=4mm} \centering
\begin{pspicture}(0,0)(7,7)
\rput(0.3,1){1}
\rput(0.3,2){2}
\rput(0.3,3){3}
\rput(0.3,4){4}
\rput(0.3,5){5}
\rput(0.3,6){6}
\rput(0.3,7){7}
\rput(1,0.2){1}
\rput(2,0.2){2}
\rput(3,0.2){3}
\rput(4,0.2){4}
\rput(5,0.2){5}
\rput(6,0.2){6}
\rput(7,0.2){7}
\psline[linecolor=gray](1,1)(1,7)\psline[linecolor=gray](1,1)(7,1)
\psline[linecolor=gray](2,1)(2,7)\psline[linecolor=gray](1,2)(7,2)
\psline[linecolor=gray](3,1)(3,7)\psline[linecolor=gray](1,3)(7,3)
\psline[linecolor=gray](4,1)(4,7)\psline[linecolor=gray](1,4)(7,4)
\psline[linecolor=gray](5,1)(5,7)\psline[linecolor=gray](1,5)(7,5)
\psline[linecolor=gray](6,1)(6,7)\psline[linecolor=gray](1,6)(7,6)
\psline[linecolor=gray](7,1)(7,7)\psline[linecolor=gray](1,7)(7,7)
\psdots(1,4)(2,3)(3,7)(4,1)(5,2)(6,6)(7,5)
\end{pspicture}  \hspace{2cm} \psset{unit=7mm}
\begin{pspicture}(0,-0.5)(7,3)
\psgrid[gridcolor=gray,griddots=4,subgriddiv=0,gridlabels=0](0,0)(7,3)
\psline(0,0)(1,1)(2,1)(3,2)(4,1)(5,2)(6,1)(7,0)
\rput(0.2,1.0){$yq^0$}
\rput(1.3,1.5){$yq^1$}
\rput(2.4,2.1){$yq^0$}
\rput(3.5,2.2){$q^0$}
\rput(4.5,2.2){$yq^1$}
\rput(5.8,1.9){$q^1$}
\rput(6.7,1){$q^0$}
\end{pspicture}
\caption{Example of the permutation $4371265$ and its image by the Françon-Viennot 
bijection. \label{fig}}
\end{figure}
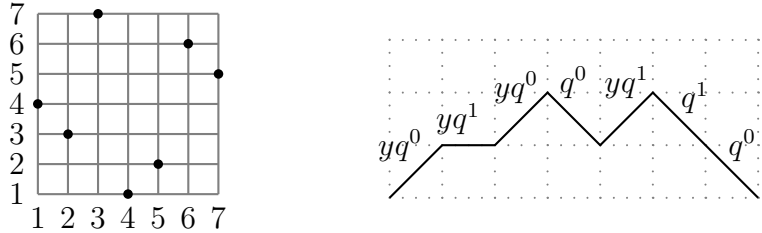

\begin{lem} \label{FV_rlmin}
   Let $\sigma\in\mathfrak{S}_n$ and $1\leq i\leq n$.
   Then $\sigma^{-1}(i)$ is a right-to-left minimum of $\sigma$ 
   if and only if the $i$th step of $\Psi_{FV}(\sigma)$ is a type 1 step.
\end{lem}

\begin{proof}
This could be done by combining the arguments of \cite{FV} and \cite{Co}. We sketch a proof introducing
ideas that will be helpful for the next lemma.

We suppose that $j=\sigma^{-1}(i)$ is a right-to-left minimum. So $j$ is an ascent, and any $v$ such that 
$i>\sigma(v)$ is such that $v<j$. The integer 31-2$(\sigma,j)$ is the number of maximal sequence of 
consecutive integers $u,u+1,\dots,v$ such that $\sigma(u)>\sigma(u+1)>\dots>\sigma(v)$, and 
$\sigma(u) > i > \sigma(v)$. Indeed, any of these sequences $u,\dots,v$ is such that $v<j$ and so
it is possible to find two consecutive elements $k,k+1$ in the sequence such that 
$\sigma(k+1)<\sigma(j)<\sigma(k)$, and these $k,k+1$ only belong to one sequence.

We call a $(\sigma,i)$-sequence a maximal sequence of consecutive integers $u,u+1,\dots,v$ such that
$\sigma(u)>\sigma(u+1)>\dots>\sigma(v)$, and $\sigma(u)\geq i > \sigma(v)$. By maximality, $u$ is a 
peak and $v$ is a valley. The number of such sequences is the difference between the number of peaks
and number of valleys among the elements of image smaller than $i$, so it is the starting height $h$ 
of the $i$th step in $\Psi_{FV}(\sigma)$. 

So with this definition, we can check that $j=\sigma^{-1}(i)$ is a right-to-left minimum of $\sigma$
if and only if $j$ is an ascent and any $(\sigma,i)$-sequence $u,u+1,\dots,v$ is such that $v<j$. 
So this is equivalent to the fact that the $i$th step of $\Psi_{FV}(\sigma)$ is a type 1 step.

\end{proof}

\begin{lem} \label{FV_rlmax}
   Let $\sigma\in\mathfrak{S}_n$, and $1\leq i\leq n$. We suppose $\sigma^{-1}(i)<n$.
   Then $\sigma^{-1}(i)$ is a right-to-left maximum of $\sigma$ if and only if
   \begin{itemize}
      \item the $i$th step of $\Psi_{FV}(\sigma)$ it is a type 2 step,
      \item any type 1 step is to the left of the $i$th step.
   \end{itemize}
\end{lem}

\begin{proof}
We keep the definition of $(\sigma,i)$-sequence as in the previous lemma. First we suppose that
$\sigma^{-1}(i)$ is a right-to-left maximum strictly smaller than $n$, and we check that the two
points are satisfied. If $\sigma^{-1}(j)$ is a right-to-left minimum, then $i>j$, so the second 
point is satisfied. A right-to-left maximum is a descent, so the $i$th step is $\rightarrow$ or 
$\searrow$ with weight $q^g$. We have to show $g=h-1$. Since $\sigma^{-1}(i)$ is a right-to-left
maximum, there is no $(\sigma,i)$-sequence $u<\dots<v$ with $\sigma^{-1}(i)<u$. So there is one
$(\sigma,i)$-sequence $u<\dots<v$ such that $u\leq \sigma^{-1}(i)<v$, and the $h-1$ other ones
contains only integers strictly smaller than $\sigma^{-1}(i)$. So the $i$th step of $\Psi_{FV}(\sigma)$
has weight $q^{h-1}$.

\smallskip

Reciprocally, we suppose that the two points above are satisfied. There are $h-1$ $(\sigma,i)$-sequence
containing integers strictly smaller than $\sigma^{-1}(i)$. Since $\sigma^{-1}(i)$ is a descent, 
the $h$th $(\sigma,i)$-sequence $u<\dots<v$ is such that $u\leq \sigma^{-1}(i)<v$. So there is
no $(\sigma,i)$-sequence $u<\dots<v$ such that $\sigma^{-1}(i)<u$.

If we suppose that $i$ is not a right-to-left maximum, there would exist $k>i$ such that 
$\sigma^{-1}(k)>\sigma^{-1}(i)$. We take the minimal $k$ satisfying this property. Then the images
of $\sigma^{-1}(k)+1,\dots,n$ are strictly greater than $k$, otherwise there would exist
$\ell>\sigma^{-1}(k)$ such that $\sigma(\ell)>i>\sigma(\ell+1)$. But then $\sigma^{-1}(k)$ 
would be a right-to-left minimum and this would contradict the second point that we assumed to
be satisfied.
\end{proof}

In Theorem~\ref{histoires} we have seen that $\bar Z_N$ is a generating function of Laguerre
histories, and the bijection $\Psi_{FV}$ together with the two lemmas above give our second new
combinatorial interpretation of $\bar Z_N$.

\begin{thm} \label{main} 
We have:
\begin{equation}
    \bar Z_N = \sum_{\sigma\in\mathfrak{S}_{N+1}}  
    \alpha^{s(\sigma)-1} \beta^{t(\sigma)-1} y^{\asc(\sigma)-1}q^{\hbox{\scriptsize \mot}(\sigma)},
\end{equation}
where we use the statistics in Definition \ref{stat_per2} above.
\end{thm}

For example, in Figure~\ref{ex_FV} we have a permutation $\sigma$ such that
\[ \alpha^{s(\sigma)-1} \beta^{t(\sigma)-1} y^{\asc(\sigma)-1}q^{\hbox{\scriptsize \mot}(\sigma)}= 
   \alpha^2\beta^3y^5q^7.
\]
Indeed $\Psi_{FV}(H)$ has total weight $y^5q^7$, has four type 1 steps and two type 2 steps to the
right of the type 1 steps.

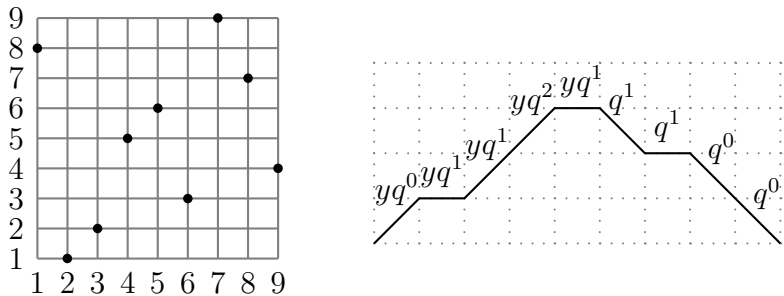
\begin{figure}[h!tp]\psset{unit=4mm} \centering
\begin{pspicture}(0,0)(9,9)
\rput(0.3,1){1}
\rput(0.3,2){2}\rput(0.3,3){3}\rput(0.3,4){4}\rput(0.3,5){5}\rput(0.3,6){6}
\rput(0.3,7){7}\rput(0.3,8){8}\rput(0.3,9){9}\rput(1,0.2){1}\rput(2,0.2){2}
\rput(3,0.2){3}\rput(4,0.2){4}\rput(5,0.2){5}\rput(6,0.2){6}\rput(7,0.2){7}
\rput(8,0.2){8}\rput(9,0.2){9}
\psline[linecolor=gray](1,1)(1,9)\psline[linecolor=gray](1,1)(9,1)
\psline[linecolor=gray](2,1)(2,9)\psline[linecolor=gray](1,2)(9,2)
\psline[linecolor=gray](3,1)(3,9)\psline[linecolor=gray](1,3)(9,3)
\psline[linecolor=gray](4,1)(4,9)\psline[linecolor=gray](1,4)(9,4)
\psline[linecolor=gray](5,1)(5,9)\psline[linecolor=gray](1,5)(9,5)
\psline[linecolor=gray](6,1)(6,9)\psline[linecolor=gray](1,6)(9,6)
\psline[linecolor=gray](7,1)(7,9)\psline[linecolor=gray](1,7)(9,7)
\psline[linecolor=gray](8,1)(8,9)\psline[linecolor=gray](1,8)(9,8)
\psline[linecolor=gray](9,1)(9,9)\psline[linecolor=gray](1,9)(9,9)
\psdots(1,8)(2,1)(3,2)(4,5)(5,6)(6,3)(7,9)(8,7)(9,4)
\end{pspicture}  
\hspace{1cm} \psset{unit=6mm}
\begin{pspicture}(0,-1)(9,4)
\psgrid[gridcolor=gray,griddots=4,subgriddiv=0,gridlabels=0](0,0)(9,4)
\psline(0,0)(1,1)(2,1)(3,2)(4,3)(5,3)(6,2)(7,2)(8,1)(9,0)
\rput(0.5,1.2){$yq^0$}
\rput(1.5,1.6){$yq^1$}
\rput(2.5,2.2){$yq^1$}
\rput(3.5,3.2){$yq^2$}
\rput(4.6,3.6){$yq^1$}
\rput(5.5,3.2){$q^1$}
\rput(6.5,2.6){$q^1$}
\rput(7.7,2.1){$q^0$}
\rput(8.7,1.1){$q^0$}
\end{pspicture}
\caption{\label{ex_FV}
The permutation $\sigma=812563974$ and its image by $\Psi_{FV}$.}
\end{figure}

\begin{rem}
We have mentioned in the introduction that the non-normalized probability of a particular 
state of the PASEP is a product $\bra{W}t_1\dots t_N\ket{V}$. It is known \cite{CW1} that 
in the combinatorial interpretation \eqref{ZTP}, this state of the PASEP corresponds to 
permutation tableaux of a given shape. It is also known \cite{CW1} that in the combinatorial 
interpretation \eqref{depart}, this state of the PASEP corresponds to permutations with a given 
set of weak exceedances (namely, $i+1$ is a weak exceedance if and only if $t_i=D$, {\it i.e.} 
the $i$th site is occuppied). It is also possible to give such criterions for the new 
combinatorial interpretations of Theorems \ref{histoires} and \ref{main}, by following 
the weak exceedances set through the bijections we have used. More precisely, in the first 
case the term $\bra{W}t_1\dots t_N\ket{V}$ is the generating function of Laguerre histories 
$H$ such that $t_i=D$ if and only if the $(N+1-i)$th step in $H$ is either a step $\rightarrow$ 
with weight $yq^i$ or a step $\searrow$. In the second case, the term $\bra{W}t_1\dots t_N\ket{V}$
is the generating function of permutations $\sigma$ such that $t_i=D$ if and only if 
$\sigma^{-1}(N+1-i)$ is a double ascent or a peak.
\end{rem}

%%%%%%%%%%%%%%%%%%%%%%%%%%%%%%%%%%%%%%%%%%%%%%%%%%%%%%%%%%%%%%%%%%%%%%%
\section{A first combinatorial derivation of $Z_N$ using lattice paths}
%%%%%%%%%%%%%%%%%%%%%%%%%%%%%%%%%%%%%%%%%%%%%%%%%%%%%%%%%%%%%%%%%%%%%%%
\label{paths}

In this section, we give the first proof of Theorem~\ref{Z_th}.

\smallskip

We consider the set $\mathfrak{P}_N$ of weighted Motzkin paths of length $N$ such that:
\begin{itemize}
\item the weight of a step $\nearrow$ starting at height $h$ is $q^i-q^{i+1}$ 
      for some $i\in\{0,\dots,h\}$,
\item the weight of a step $\rightarrow$ starting at height $h$ is either $1+y$ 
      or $(\alp+y\bet)q^h$,
\item the weight of a step $\searrow$ starting at height $h$ is either $y$ or 
      $-y\alp\bet q^{h-1}$.
\end{itemize}
The sum of weights of elements in $\mathfrak{P}_N$ is $(1-q)^N Z_N$ because the 
weights sum to the ones in \eqref{Z_w}. We stress that on the combinatorial
point of view, it will be important to distinguish $(h+1)$ kinds of step $\nearrow$ 
starting at height $h$, instead of one kind of step $\nearrow$ with weight $1-q^{h+1}$. 

\smallskip

We will show that each element of $\mathfrak{P}_N$ bijectively corresponds to a pair of 
weighted Motzkin paths. The first path (respectively, second path) belongs to a set whose generating 
function is $R_{N,n}(y,q)$ (respectively, $B_n(\alp,\bet,y,q)$) for some $n\in\{0,\dots,N\}$. 
Following this scheme, our first combinatorial proof of \eqref{Z} is a consequence of
Propositions~\ref{RNn}, \ref{Bn}, and \ref{decomp} below.

\subsection{The lattice paths for $R_{N,n}(y,q)$}

Let $\mathfrak{R}_{N,n}$ be the set of weighted Motzkin paths of length $N$ such that:
\begin{itemize}
\item the weight of a step $\nearrow$ starting at height $h$ is either $1$ or $-q^{h+1}$,
% $q^i-q^{i+1}$ for some $i\in\{0,\dots,h\}$,
\item the weight of a step $\rightarrow$ starting at height $h$ is either $1+y$ or $q^h$,
\item the weight of a step $\searrow$ is $y$,
\item there are exactly $n$ steps $\rightarrow$ weighted by a power of $q$.
\end{itemize}

In this subsection we prove the following:

\begin{prop} \label{RNn}
The sum of weights of elements in $\mathfrak{R}_{N,n}$ is $R_{N,n}(y,q)$.
\end{prop}

This can be obtained with the methods used in \cite{CJPR,MJV}, and the result is a 
consequence of the Lemmas~\ref{decomp2}, \ref{motzbi} and \ref{core} below.

\begin{lem} \label{decomp2}
There is a weight-preserving bijection between $\mathfrak{R}_{N,n}$, and the pairs $(P,C)$
such that for some $i\in\{0,\dots,\lfloor\frac{N-n}2\rfloor\}$, 
\begin{itemize}
\item $P$ is a Motzkin prefix of length $N$ and final height $n+2i$, with a 
weight $1+y$ on every step $\rightarrow$, and a weight $y$ on every step $\searrow$,
\item $C$ is a Motzkin path of length $n+2i$, such that
\begin{equation} \label{def_core}
\parbox{13cm}{
\begin{itemize}
\item the weight of a step $\nearrow$ starting at height $h$ is $1$ or $-q^h$, 
\item the weight of a step $\rightarrow$ starting at height $h$ is $q^h$,
\item the weight of a step $\searrow$ is $1$,
\item there are exactly $n$ steps $\rightarrow$, and no steps $\nearrow\searrow$ both with weights 1.
\end{itemize}}
\end{equation}
\end{itemize}
\end{lem}

\begin{proof} This is a direct adaptation of \cite[Lemma~1]{CJPR}.
\end{proof}

\begin{lem} \label{motzbi}
The generating function of Motzkin prefixes of length $N$ and final height $n+2i$, with a 
weight $1+y$ on every step $\rightarrow$, and a weight $y$ on every step $\searrow$, is 
\[   
   \sum_{j=0}^{N-n-2i}y^j\left( \tbinom Nj \tbinom N{n+2i+j}-
   \tbinom N{j-1} \tbinom N{n+2i+j+1}  \right).
\]
\end{lem}

\begin{proof} This was given in \cite[Proposition~4]{CJPR}.
\end{proof}

\begin{lem} \label{core}
The sum of weights of Motzkin paths of length $n+2i$ satisfying properties \eqref{def_core}
above is $(-1)^iq^{\binom{i+1}2}\qbin {n+i}i$.
\end{lem}

\begin{proof} A bijective proof was given in \cite[Lemmas~3, 4]{MJV}.
\end{proof}

Some precisions are in order. In \cite{CJPR} and \cite{MJV}, we obtained the formula \eqref{cas1}
which is the special case $\alpha=\beta=1$ in $Z_N$, and is the $N$th moment of the $q$-Laguerre
polynomials mentioned in Remark \ref{rem_poly}. Since $Z_N$ is also very closely related 
with these polynomials (see Section~\ref{ALSC}) it is not surprising that some steps are in
common between these previous results and the present ones. See also Remark \ref{comp1} below.

\subsection{The lattice paths for $B_n(\alp,\bet,y,q)$}

Let $\mathfrak{B}_n$ be the set of weighted Motzkin paths of length $n$ such that:
\begin{itemize}
\item the weight of a step $\nearrow$ starting at height $h$ is either $1$ or $-q^{h+1}$,
% $q^i-q^{i+1}$ for some $i\in\{0,\dots,h\}$,
\item the weight of a step $\rightarrow$ starting at height $h$ is $(\alp+y\bet)q^h$,
\item the weight of a step $\searrow$ starting at height $h$ is $-y\alp\bet q^{h-1}$.
\end{itemize}

In this section we prove the following:

\begin{prop} \label{Bn}
The sum of weights of elements in $\mathfrak{B}_n$ is $B_n(\alp,\bet,y,q)$.
\end{prop}

\begin{proof} 
Let $\nu_n$ be the sum of weights of elements in $\mathfrak{B}_n$. It is
homogeneous of degree $n$ in $\alp$ and $\bet$ since each step $\rightarrow$ has degree 1
and each pair of steps $\nearrow$ and $\searrow$ has degree 2. By comparing the weights for
paths in $\mathfrak{B}_n$, and the ones in \eqref{Z_w}, we see that $\nu_n$ is the term of
$(1-q)^n Z_n$ with highest degree in $\alp$ and $\bet$. Since $\alp$ and $(1-q)\frac1\alpha$
(respectively, $\bet$ and $(1-q)\frac1\beta$) only differ by a constant, it remains only to 
show that the term of $\bar Z_n$ with highest degree in $\alpha$ and $\beta$ is 
$\sum_{k=0}^n\qbin{n}{k}\alpha^k(y\beta)^{n-k}$.

This follows from the combinatorial interpretation in Equation~\eqref{ZTP} in terms of permutation
tableaux (see Definition \ref{def_PT}). In the term of $\bar Z_n$ with highest degree in 
$\alpha$ and $\beta$, the coefficient of $\alpha^k\beta^{n-k}$ is obtained by counting
permutations tableaux of size $n+1$, with $n-k+1$ unrestricted rows, $k$ 1s in the first 
row. Such permutation tableaux have $n-k+1$ rows, $k$ columns, and contain no 0. They are 
in bijection with the Young diagrams that fit in a $k\times(n-k)$ box and give a factor 
$\qbin nk$.
\end{proof}

We can give a second proof in relation with orthogonal polynomials.

\begin{proof}
It is a consequence of properties of the Al-Salam-Carlitz orthogonal polynomials 
$U_k^{(a)}(x)$, defined by the recurrence \cite{ASCa,KoSw98}:
\begin{equation}
U_{k+1}^{(a)}(x) = xU_{k}^{(a)}(x) + (a+1)q^kU_{k}^{(a)}(x) + a(q^k-1) q^{k-1}U_{k-1}^{(a)}(x).
\end{equation}
Indeed, from Proposition~\ref{mu_ortho} the sum of weights of elements in $\mathfrak{B}_n$ is the
$n$th moment of the orthogonal polynomial sequence $\{P_k(x) \}_{k\geq 0} $ defined by 
\begin{equation}
P_{k+1}(x) = xP_k(x) + (\alp+y\bet)q^kP_k + (q^k-1)y\alp\bet q^{k-1}P_{k-1}.
\end{equation}
We have $P_k(x)=(y\bet)^k U_k^{(a)}( x  (y\bet)^{-1} )$ where $a=\alp(y\bet)^{-1}$, and the $n$th 
moment of the sequence $\{U_k^{(a)}(x)\}_{k\geq 0}$ is 
$\sum_{j=0}^k\qbin kj a^j$ (see \S 5 in \cite{ASCa}, or the article of D. Kim \cite[Section~3]{DK}
for a combinatorial proof). Then we can derive the moments of $\{P_k(x)\}_{k\geq0}$, and this 
gives a second proof of Proposition~\ref{Bn}.
\end{proof}

\smallskip

Another possible proof would be to write the generating function $\sum_{n=0}^\infty \nu_n z^n$
as a continued fraction with the usual methods \cite{Fla}, use a limit case of identity (19.2.11a) 
in \cite{CU} to relate this generating function with a basic hypergeometric series and then expand 
the series.

\subsection{The decomposition of lattice paths}

Let $\mathfrak{R}^*_{N,n}$ be defined exactly as $\mathfrak{R}_{N,n}$, except that the possible
weights of a step $\nearrow$ starting at height $h$ are $q^i-q^{i+1}$ with $i\in\{0,\dots,h\}$.
The sum of weights of elements in $\mathfrak{R}^*_{N,n}$ is the same as with $\mathfrak{R}_{N,n}$,
because the possible weights of a step $\nearrow$ starting at height $h$ sum to $1-q^{h+1}$.
Similarly let $\mathfrak{B}^*_n$ be defined exactly as $\mathfrak{B}_n$, except that the possible
weights of a step $\nearrow$ starting at height $h$ are $q^i-q^{i+1}$ with $i\in\{0,\dots,h\}$.

\begin{prop} \label{decomp}
There exists a weight-preserving bijection $\Phi$ between the disjoint union of 
$\mathfrak{R}^*_{N,n} \times \mathfrak{B}^*_n$ over $n\in\{0,\dots,N\}$, and $\mathfrak{P}_n$
(we understand that the weight of a pair is the product of the weights of each element).
\end{prop}

To define the bijection, we start from a pair $(H_1,H_2)\in\mathfrak{R}^*_{N,n} 
\times \mathfrak{B}^*_n$ for some $n\in\{0,\dots,N\}$ and build a path 
$\Phi(H_1,H_2)\in\mathfrak{P}_N$. Let $i\in\{1,\dots,N\}$.
\begin{itemize}
\item If the $i$th step of $H_1$ is a step $\rightarrow$ weighted by a power 
      of $q$, say the $j$th one among the $n$ such steps, then:
  \begin{itemize}
    \item  the $i$th step $\Phi(H_1,H_2)$ has the same direction as the $j$th step  of $H_2$,
    \item  its weight is the product of weights of the $i$th step of $H_1$ and the $j$th step 
       of $H_2$.
  \end{itemize}
\item Otherwise the $i$th step of $\Phi(H_1,H_2)$ has the same direction and 
  same weight as the $i$th step of $H_1$.
\end{itemize}
See Figure~\ref{decomp_ex} for an example, where the thick steps correspond to the ones in
the first of the two cases considered above. It is immediate that the total weight of $\Phi(H_1,H_2)$ 
is the product of the total weights of $H_1$ and $H_2$.

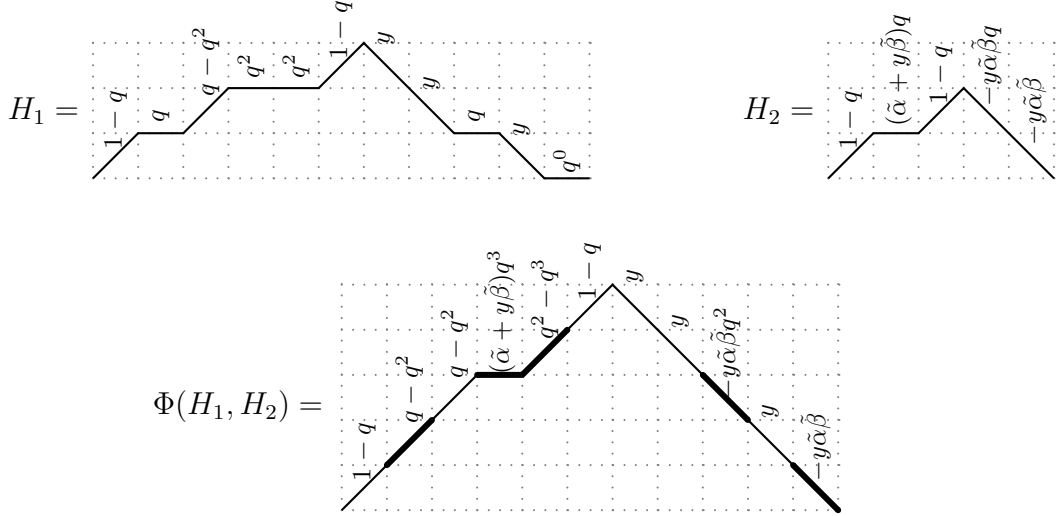
\begin{figure}[h!tp] \centering
\psset{unit=6mm}
\begin{pspicture}(-1.5,0)(11,4)
\psgrid[gridcolor=gray,griddots=4,subgriddiv=0,gridlabels=0](0,0)(11,3)
\psline(0,0)(1,1)(2,1)(3,2)(5,2)(6,3)(7,2)(8,1)(9,1)(10,0)(11,0)
\rput(0.5,1.3) {\footnotesize \begin{sideways} $1-q$ \end{sideways}}
\rput(1.5,1.4) {\footnotesize \begin{sideways} $q$ \end{sideways}}
\rput(2.5,2.6) {\footnotesize \begin{sideways} $q-q^2$ \end{sideways}}
\rput(3.5,2.5) {\footnotesize \begin{sideways} $q^2$ \end{sideways}}
\rput(4.5,2.5) {\footnotesize \begin{sideways} $q^2$ \end{sideways}}
\rput(5.5,3.3) {\footnotesize \begin{sideways} $1-q$ \end{sideways}}
\rput(6.5,3.1) {\footnotesize \begin{sideways} $y$ \end{sideways}}
\rput(7.5,2.1) {\footnotesize \begin{sideways} $y$ \end{sideways}}
\rput(8.5,1.4) {\footnotesize \begin{sideways} $q$ \end{sideways}}
\rput(9.5,1.1) {\footnotesize \begin{sideways} $y$ \end{sideways}}
\rput(10.5,0.4){\footnotesize \begin{sideways} $q^0$ \end{sideways}}
\rput(-1.1,1.5){$H_1=$}
\end{pspicture} 
\hspace{2cm}
\psset{unit=6mm}
\begin{pspicture}(-1.5,0)(5,4)
\psgrid[gridcolor=gray,griddots=4,subgriddiv=0,gridlabels=0](0,0)(5,3)
\psline(0,0)(1,1)(2,1)(3,2)(4,1)(5,0)
\rput(0.5,1.3){\footnotesize \begin{sideways} $1-q$ \end{sideways}}
\rput(1.5,2.4){\footnotesize \begin{sideways} $(\alp+y\bet)q$ \end{sideways}}
\rput(2.5,2.3){\footnotesize \begin{sideways} $1-q$ \end{sideways}}
\rput(3.5,2.5){\footnotesize \begin{sideways} $-y\alp\bet q$ \end{sideways}}
\rput(4.5,1.5){\footnotesize \begin{sideways} $-y\alp\bet $ \end{sideways}}
\rput(-1.1,1.5){$H_2=$}
\end{pspicture}
\vspace{6mm}

\begin{pspicture}(-2,0)(11,6.3)
\psgrid[gridcolor=gray,griddots=4,subgriddiv=0,gridlabels=0](0,0)(11,5)
\psline(0,0)(3,3)(4,3)(6,5)(11,0)
\rput(0.5,1.3) {\footnotesize \begin{sideways} $1-q$ \end{sideways}}
\rput(1.5,2.7) {\footnotesize \begin{sideways} $q-q^2$ \end{sideways}}
\rput(2.5,3.7) {\footnotesize \begin{sideways} $q-q^2$ \end{sideways}}
\rput(3.5,4.4) {\footnotesize \begin{sideways} $(\alp+y\bet)q^3$ \end{sideways}}
\rput(4.5,4.7) {\footnotesize \begin{sideways} $q^2-q^3$ \end{sideways}}
\rput(5.5,5.3) {\footnotesize \begin{sideways} $1-q$ \end{sideways}}
\rput(6.5,5.2) {\footnotesize \begin{sideways} $y$ \end{sideways}}
\rput(7.5,4.2) {\footnotesize \begin{sideways} $y$ \end{sideways}}
\rput(8.5,3.5) {\footnotesize \begin{sideways} $-y\alp\bet q^2$ \end{sideways}}
\rput(9.5,2.2) {\footnotesize \begin{sideways} $y$ \end{sideways}}
\rput(10.5,1.4){\footnotesize \begin{sideways} $-y\alp\bet$ \end{sideways}}
\psline[linewidth=0.8mm]{c-c}(1,1)(2,2)
\psline[linewidth=0.8mm]{c-c}(3,3)(4,3)(5,4)
\psline[linewidth=0.8mm]{c-c}(8,3)(9,2)
\psline[linewidth=0.8mm]{c-c}(10,1)(11,0)
\rput(-2.3,2.3){$\Phi(H_1,H_2)=$}
\end{pspicture} 
\caption{Example of paths $H_1$, $H_2$ and their image $\Phi(H_1,H_2)$ \label{decomp_ex}.}
\end{figure}

The inverse bijection is not as simple. Let $H\in\mathfrak{P}_N$. The method consists in reading
$H$ step by step from right to left, and building two paths $H_1$ and $H_2$ step by step so that 
at the end we obtain a pair $(H_1,H_2)\in\mathfrak{R}^*_{N,n} \times \mathfrak{B}^*_n$ for some 
$n\in\{0,\dots,N\}$. At each intermediate stage, we have built two Motzkin {\it suffixes}, 
{\it i.e.} some paths similar to Motzkin paths except that the starting height may be non-zero. 

Let us fix some notation. Let $H^{(j)}$ be the Motzkin suffix obtained by taking the $j$ last
steps of $H$. Suppose that we have already read the $j$ last steps of $H$, and built two Motzkin 
suffixes $H_1^{(j)}$ and $H_2^{(j)}$. We describe how to iteratively obtain $H_1^{(j+1)}$ and 
$H_2^{(j+1)}$. Note that $H_1^{(0)}$ and $H_2^{(0)}$ are empty paths. Let $h$, $h'$, and $h''$
be the respective starting heights of $H^{(j)}$, $H_1^{(j)}$ and $H_2^{(j)}$.

This iterative construction will have the following properties, as will be immediate 
from the definition below:
\begin{itemize}
\item $H_1^{(j)}$ has length $j$, and the length of $H_2^{(j)}$ is the number of steps 
      $\rightarrow$ weighted by a power of $q$ in $H_1^{(j)}$.
\item We have $h=h'+h''$.
\item The map $\Phi$ as we described it can also be defined in the same way for Motzkin 
      suffixes, and is such that $H^{(j)} = \Phi(H_1^{(j)},H_2^{(j)})$.
\end{itemize}
To obtain $H_1^{(j+1)}$ and $H_2^{(j+1)}$, we read the $(j+1)$th step 
in $H$ starting from the right, and add steps to the left of $H_1^{(j)}$ and $H_2^{(j)}$ 
according to the following rules:

\begin{center}
\begin{tabular}{c|c|c}
  step read in $H$   & step added to  $H_1^{(j)}$   &  step added to $H_2^{(j)}$ \\[2mm]
\hline {\huge \phantom l}\hspace{-3mm}
$\searrow$  $-y\alp\bet q^h$  &  $\rightarrow$ $q^{h'}$ & $\searrow$ $-y\alp\bet q^{h''}$ \\[2mm]
\hline {\huge \phantom l}\hspace{-3mm}
$\searrow$  $y$  &  $\searrow$ $y$  &       \\[2mm]
\hline {\huge \phantom l}\hspace{-3mm}
$\rightarrow$  $1+y$  &  $\rightarrow$ $1+y$  &       \\[2mm]
\hline {\huge \phantom l}\hspace{-3mm}
$\rightarrow$  $(\alp+y\bet)q^h$  &  $\rightarrow$ $q^{h'}$  &  $\rightarrow$  
$(\alp+y\bet)q^{h''}$\\[2mm]
\hline {\huge \phantom l}\hspace{-3mm}
$\nearrow$  $q^i-q^{i+1}$ \hbox{ with } $i<h'$  &  $\nearrow$  $q^i-q^{i+1}$  &     \\[2mm]
\hline {\huge \phantom l}\hspace{-3mm}
$\nearrow$  $q^i-q^{i+1}$ \hbox{ with } $i\geq h'$  &  $\rightarrow$  $q^{h'}$ &
$\nearrow$ $q^{i-h'}-q^{i+1-h'}$
\\[2mm]
\end{tabular}
\end{center}

We can also iteratively check the following points.
\begin{itemize}
\item With this construction $H_1^{(j+1)}$ and $H_2^{(j+1)}$ are indeed Motzkin suffixes.
      This is because we add a step $\nearrow$ to $H_1^{(j)}$ only in the case where $i<h'$, 
      hence $h'>0$. And we add a step $\nearrow$ to $H_2^{(j)}$ only in the case where $i\geq h'$,
      hence $h''>0$ (since $h=h'+h''>i$).
\item The paths $H_1^{(j+1)}$ and $H_2^{(j+1)}$ are respectively suffixes of an element in
      $\mathfrak{R}_{N,n}$ and $\mathfrak{B}_n$ for some $n\in\{0,\dots,N\}$,
      {\it i.e.} the weights are valid.
\item The set of rules we have given is the only possible one such that for any $j$ we have
      $H^{(j)} = \Phi(H_1^{(j)},H_2^{(j)})$.
\end{itemize}
It follows that $(H^{(N)}_1,H^{(N)}_2)\in\mathfrak{R}^*_{N,n}\times\mathfrak{B}^*_n$ for some 
$n\in\{0,\dots,N\}$, these paths are such that $\Phi(H^{(N)}_1,H^{(N)}_2)=H$, and it is 
the only pair of paths satisfying these properties. There are details to check, but we have 
a full description of $\Phi$ and of the inverse map $\Phi^{-1}$. See Figure~\ref{inv_ex} for
an example of the Motzkin suffixes we consider.

\begin{figure}[h!tp] \centering
\psset{unit=5mm}
\begin{pspicture}(-1.5,0)(5,4.5)
\psgrid[gridcolor=gray,griddots=4,subgriddiv=0,gridlabels=0](0,0)(5,4)
\psline(0,2)(1,3)(2,2)(3,2)(4,1)(5,0)
\rput(0.5,3.7) {\footnotesize \begin{sideways} $q-q^2$ \end{sideways}}
\rput(1.5,3.8) {\footnotesize \begin{sideways} $-y\alp\bet q^2$ \end{sideways}}
\rput(2.5,3.6) {\footnotesize \begin{sideways} $(\alp+y\bet)q^2$ \end{sideways}}
\rput(3.5,2.3) {\footnotesize \begin{sideways} $y$ \end{sideways}}
\rput(4.5,1.8) {\footnotesize \begin{sideways} $-y\alp\bet$ \end{sideways}}
\rput(-1.5,1.5){$H^{(j)}=$}
\end{pspicture} \psset{unit=5mm} \hspace{1.4cm}
\begin{pspicture}(-1.5,0)(5,4.5)
\psgrid[gridcolor=gray,griddots=4,subgriddiv=0,gridlabels=0](0,0)(5,4)
\psline(0,1)(1,1)(2,1)(3,1)(4,0)(5,0)
\rput(0.5,1.6) {\footnotesize \begin{sideways} $q^1$ \end{sideways}}
\rput(1.5,1.6) {\footnotesize \begin{sideways} $q^1$ \end{sideways}}
\rput(2.5,1.6) {\footnotesize \begin{sideways} $q^1$ \end{sideways}}
\rput(3.5,1.2) {\footnotesize \begin{sideways} $y$ \end{sideways}}
\rput(4.5,0.8) {\footnotesize \begin{sideways} $q^0$ \end{sideways}}
\rput(-1.5,1.5){$H_1^{(j)}=$}
\end{pspicture} \psset{unit=5mm} \hspace{1.4cm}
\begin{pspicture}(-1.5,0)(4,4.5)
\psgrid[gridcolor=gray,griddots=4,subgriddiv=0,gridlabels=0](0,0)(4,4)
\psline(0,1)(1,2)(2,1)(3,1)(4,0)
\rput(0.5,2.5) {\footnotesize \begin{sideways} $1-q$ \end{sideways}}
\rput(1.5,2.6) {\footnotesize \begin{sideways} $-y\alp\bet q$ \end{sideways}}
\rput(2.5,2.6) {\footnotesize \begin{sideways} $(\alp+y\bet)q$ \end{sideways}}
\rput(3.5,1.5) {\footnotesize \begin{sideways} $-y\alp\bet$ \end{sideways}}
\rput(-1.5,1.5){$H_2^{(j)}=$}
\end{pspicture} 
\caption{\label{inv_ex}
Example of Motzkin suffixes used to define $\Phi^{-1}$.
}
\end{figure}
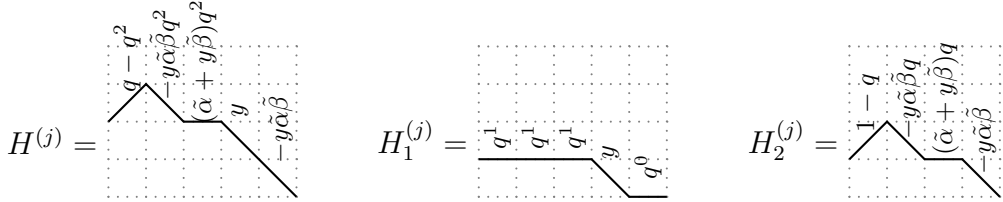

Before ending this subsection we can mention another argument to show that $\mathfrak{P}_N$ 
and the disjoint union of $\mathfrak{R}^*_{N,n}\times\mathfrak{B}^*_n$ have the same cardinal. 
Thus we could just focus on the surjectivity of the map $\Phi$ and avoid making the inverse 
map explicit. The argument uses the notion of {\it histories} \cite{XGV} and their link with 
classical combinatorial objects, as we have seen in the previous section with Laguerre histories.
As an unweighted set, $\mathfrak{P}_N$ is a set of colored Motzkin paths, with two possible colors
on the steps $\rightarrow$ or $\searrow$, and $h+1$ possible colors for a step $\nearrow$ starting
at height $h$. So $\mathfrak{P}_N$ is in bijection with colored involutions $I$ on the set 
$\{1,\dots,N\}$, such that there are two possible colors on each fixed point or each arch 
(orbit of size 2). So they are also in bijection with pairs $(I_1,I_2)$ such that for some 
$n\in\{0,\dots,N\}$:
\begin{itemize}
\item $I_1$ is an involution on $\{1,\dots,N\}$ with two possible colors on the fixed points 
      (say, blue and red), and having exactly $n$ red fixed points,
\item $I_2$ is an involution on $\{1,\dots,n\}$.
\end{itemize}
Using histories again, we see that the number of such pairs $(I_1,I_2)$ is the cardinal of
$\mathfrak{R}^*_{N,n}\times\mathfrak{B}^*_n$.

\begin{rem} Note that considering $\mathfrak{P}_N$ as an unweighted set is not equivalent
to setting the various parameters to 1. For example the two possible colors for the
horizontal steps correspond to the possible weights $1+y$ or $(\alp+y\bet)q^i$.
This bijection using colored involution is not weight-preserving but it might be possible
to have a weight-preserving version of it for some adequate statistics on the colored 
involutions.
\end{rem}

\begin{rem} \label{comp1}
The decomposition $\Phi$ is the key step in our first proof of Theorem~\ref{Z_th}. This makes
the proof quite different from the one in the case $\alpha=\beta=1$ \cite{CJPR}, even though we have 
used results from \cite{CJPR} to prove an intermediate step (namely Proposition~\ref{RNn}).
Actually it might be possible to have a direct adaptation of the case $\alpha=\beta=1$ \cite{CJPR} 
to prove Theorem~\ref{Z_th}, but it should give rise to many computational steps. In contrast
our decomposition $\Phi$ explains the formula for $Z_N$ as a sum of products.
\end{rem}

%%%%%%%%%%%%%%%%%%%%%%%%%%%%%%%%%%%%%%%%%%%%%%%%%%%%%%%%%%%%%%
\section{A second derivation of $Z_N$ using the Matrix Ansatz}
%%%%%%%%%%%%%%%%%%%%%%%%%%%%%%%%%%%%%%%%%%%%%%%%%%%%%%%%%%%%%%
\label{rooks}

In this section we build on our previous work \cite{MJV} to give a second
proof of \eqref{Z}. In this reference we define the operators
\begin{equation} \label{def_DE}
  \hat D = \frac{q-1}q D + \frac 1q I \qquad \hbox{and} 
  \qquad \hat E = \frac{q-1}q E + \frac 1q I,
\end{equation}
where $I$ is the identity. Some immediate consequences are
\begin{equation} \label{comrel}
 \hat D \hat E - q \hat E \hat D = \frac{1-q}{q^2}, \qquad \bra{W} \hat E = 
 -\frac \alp q\bra{W}, \qquad \hbox{and}\quad \hat D\ket{V} = -\frac \bet q \ket{V},
\end{equation}
where $\alp$ and $\bet$ are defined as in the previous sections. While the normal
ordering problem for $D$ and $E$ leads to permutation tableaux, for $\hat D$ and
$\hat E$ it leads to {\it rook placements} as was shown for example in \cite{AV}.
The combinatorics of rook placements lead to the following proposition.

\begin{prop} \label{hats} We have:
\begin{equation}
    \bra{W}(qy\hat D+q\hat E)^k\ket{V} = 
     \sum_{\substack{i+j \leq k  \\   i+j \equiv k \hbox{\scriptsize \, mod }2    }     } 
     \qbin{i+j}{i} (-\alp)^i (-y\bet)^j M_{\frac{k-i-j}2 , k}
\end{equation}
where
\begin{equation}
     M_{\ell,k} = y^{\ell} \sum_{u=0}^\ell (-1)^u q^{\binom{u+1}2} \qbin{k-2\ell+u}u
     \left( \binom{k}{\ell-u} - \binom{k}{\ell-u-1} \right).
\end{equation}
\end{prop}

\begin{proof} This is a consequence of results in \cite{MJV} (see Section~2, Corollary~1,
Proposition~12). We also give here a self-contained recursive proof. We write the normal
form of $(yq\hat D+q\hat E)^k$ as:
\begin{equation} \label{rec1}
  (yq\hat D+q\hat E)^k = \sum_{i,j\geq0} d^{(k)}_{i,j} (q\hat E)^i(qy\hat D)^j.
\end{equation}
From the commutation relation in \eqref{comrel} we obtain:
\begin{equation} \label{rec2}
 (qy\hat D)^j(q\hat E) = q^j(q\hat E)(qy\hat D)^j + y(1-q^j)(qy\hat D)^{j-1}.
\end{equation}
If we multiply \eqref{rec1} by $yq\hat D+q\hat E$ to the right, using \eqref{rec2} we can get a
recurrence relation for the coefficients $d^{(k)}_{i,j}$, which reads:
\begin{equation}
d^{(k+1)}_{i,j} = d^{(k)}_{i,j-1} + q^j d^{(k)}_{i-1,j} + y(1-q^{j+1}) d^{(k)}_{i,j+1}.
\end{equation}
The initial case is that $d^{(0)}_{i,j}$ is 1 if $(i,j)=(0,0)$ and 0 otherwise.
It can be directly checked that the recurrence is solved by:
\begin{equation}
d^{(k)}_{i,j} = \qbin{i+j}i M_{\frac{k-i-j}2,k}
\end{equation}
where we understand that $M_{\frac{k-i-j}2,k}$ is $0$ when $k-i-j$ is not even. More precisely,
if we let $e^{(k)}_{i,j} = \qbin{i+j}i M_{\frac{k-i-j}2,k}$ then we have:
\begin{equation}
e^{(k)}_{i,j-1} + q^j e^{(k)}_{i-1,j} = \qbin{i+j}i M_{\frac{k-i-j+1}2,k},
\end{equation}
and also
\begin{equation}
y(1-q^{j+1}) e^{(k)}_{i,j+1} = y (1-q^{i+j+1}) \qbin{i+j}i M_{\frac{k-i-j-1}2,k}.
\end{equation}
So to prove $d^{(k)}_{i,j}=e^{(k)}_{i,j}$ it remains only to check that 
\begin{equation}
M_{\frac{k-i-j+1}2,k} + y (1-q^{i+j+1}) M_{\frac{k-i-j-1}2,k} = M_{\frac{k-i-j+1}2,k+1}.
\end{equation}
See for example \cite[Proposition~12]{MJV} (actually this recurrence already appeared more than
fifty years ago in the work of Touchard, see {\it loc. cit.} for precisions).
\end{proof}

Now we can give our second proof of Theorem~\ref{Z_th}.

\begin{proof} From \eqref{ZN} and \eqref{def_DE} we have that $(1-q)^N Z_N$ is equal to
\[   
  \bra{W}((1+y)I-qy\hat D - q \hat E )^N\ket{V} = 
  \sum_{k=0}^{N} \binom{N}{k}(1+y)^{N-k}(-1)^k \bra{W}(qy\hat D + q\hat E)^k\ket{V}.
\]
So, from Proposition~\ref{hats} we have:
\[  (1-q)^N Z_N = \sum_{k=0}^N 
    \sum_{\substack{i+j \leq k  \\   i+j \equiv k \hbox{\scriptsize \, mod }2    }     } 
    \qbin{i+j}i \alp^i(y\bet)^j\binom Nk (1+y)^{N-k}
    M_{\frac{k-i-j}2,k}
\]
(the $(-1)^k$ cancels with a $(-1)^{i+j}$). Setting $n=i+j$, we have:
\begin{eqnarray*}
  (1-q)^N Z_N   % & = &
%    \sum_{n=0}^N \sum_{\substack{ 0\leq k\leq n \\ k \equiv n \hbox{ mod }2   } }
%    B_n(\alp,\bet,y,q) \binom Nk 2^{N-k} (1-q)^{\frac{k-n}2}M_{\frac{k-n}2,k} \\
  & = & \sum_{n=0}^N  B_n(\alp,\bet,y,q)
          \sum_{\substack{ n\leq k\leq N \\ k \equiv n \hbox{\scriptsize \, mod }2   } }
       \binom Nk (1+y)^{N-k}  M_{\frac{k-n}2,k}.
\end{eqnarray*}
So it remains only to show that the latter sum is $R_{N,n}(y,q)$. If we change the indices
so that $k$ becomes $n+2k$, this sum is:
\[
  \sum_{k=0}^{\lfloor \frac{N-n}2 \rfloor} \tbinom{N}{n+2k}(1+y)^{N-n-2k} y^{k}
  \sum_{i=0}^k (-1)^i q^{\binom{i+1}2} \qbin{n+i}i
  \left(  \tbinom{n+2k}{k-i} - \tbinom{n+2k}{k-i-1}     \right)
\]
\[
  = \sum_{i=0}^{\lfloor \frac{N-n}2 \rfloor} (-y)^i q^{\binom{i+1}2} \qbin{n+i}i
   \sum_{k=i}^{\lfloor \frac{N-n}{2} \rfloor } y^{k-i}\tbinom N{n+2k} (1+y)^{N-n-2k}
   \left( \tbinom{n+2k}{k-i} - \tbinom{n+2k}{k-i-1}   \right).
\]
We can simplify the latter sum by Lemma~\ref{idbinl} below and obtain $R_{N,n}(y,q)$.
This completes the proof.
\end{proof}

\begin{lem} \label{idbinl}For any $N,n,i\geq0$ we have:
\begin{equation} \label{idbin}
\begin{split}
   \sum_{k=i}^{\lfloor \frac{N-n}{2} \rfloor } y^{k-i} \binom N{n+2k} (1+y)^{N-n-2k}
     & \left( \tbinom{n+2k}{k-i} - \tbinom{n+2k}{k-i-1}   \right)  \\
     & =  
     \sum_{j=0}^{N-n-2i}y^j\left( \tbinom Nj \tbinom N{n+2i+j}-
     \tbinom N{j-1} \tbinom N{n+2i+j+1}  \right).
\end{split}
\end{equation}
\end{lem}

\begin{proof} As said in Lemma~\ref{motzbi}, the right-hand side of \eqref{idbin} is
the number of Motzkin prefixes of length $N$, final height $n+2i$, and a weight $1+y$ on each 
step $\rightarrow$ and $y$ on each step $\searrow$. Similarly, $y^{k-i}(\tbinom{n+2k}{k-i} - 
\tbinom{n+2k}{k-i-1})$ is the number of Dyck prefixes of length $n+2k$ and final height $n+2i$, 
with a weight $y$ on each step $\searrow$. From these two combinatorial interpretations it is
straightforward to obtain a bijective proof of \eqref{idbin}. Each Motzkin prefix is built from
a shorter Dyck prefix with the same final height, by choosing where are the $N-n-2k$ steps
$\rightarrow$.
\end{proof}

\begin{rem} \label{comp2}
All the ideas in this proof were present in \cite{MJV} where we obtained the case $\alpha=\beta=1$. 
The particular case was actually more difficult to prove because several $q$-binomial and binomial
simplifications were needed. In particular, it is natural to ask if the formula in \eqref{cas1} 
for $Z_N|_{\alpha=\beta=1}$ can be recovered from the general expression in Theorem~\ref{Z_th},
and the (affirmative) answer is essentially given in \cite{MJV} (see also Subsection \ref{simpli} 
below for a very similar simplification).
\end{rem}

%%%%%%%%%%%%%%%%%%%%%%%%%%%%%%%%%%%%%%%%%%%%%%%%%
\section{Moments of Al-Salam-Chihara polynomials}
%%%%%%%%%%%%%%%%%%%%%%%%%%%%%%%%%%%%%%%%%%%%%%%%%
\label{ALSC}

The link between the PASEP and Al-Salam-Chihara orthogonal polynomials
$Q_n(x;a,b\mid q)$ was described in \cite{TS}. These polynomials, denoted by $Q_n(x)$ 
when we don't need to specify the other parameters, are defined by the recurrence \cite{KoSw98}:
\begin{equation} \label{recASC}
 2xQ_n(x) = Q_{n+1}(x) + (a+b)q^nQ_n(x) + (1-q^n)(1-abq^{n-1})Q_{n-1}(x)
\end{equation}
together with $Q_{-1}(x)=0$ and $Q_0(x)=1$. They are the most general orthogonal
sequence that is a convolution of two orthogonal sequences \cite{ASCh}. They are obtained 
from Askey-Wilson polynomials $p_n(x;a,b,c,d\mid q)$ by setting $c=d=0$ \cite{KoSw98}.

\subsection{Closed formulas for the moments}
Let $\tilde Q_n(x) = Q_n(\frac x2-1; \alp , \bet \mid q)$, where $\alp = (1-q)\frac1\alpha-1 $ 
and $\bet = (1-q)\frac1\beta-1$ as before. From now on we suppose that $a=\alp$ and $b=\bet$ (note
that $a$ and $b$ are generic if $\alpha$ and $\beta$ are). The recurrence for these shifted 
polynomials is:
\begin{equation}
 x\tilde Q_n(x) = \tilde Q_{n+1}(x) + (2+\alp q^n+ \bet q^n)\tilde Q_n(x) 
                  + (1-q^n)(1-\alp\bet q^{n-1})\tilde Q_{n-1}(x).
\end{equation}
From Proposition~\ref{mu_ortho}, the $N$th moment of the orthogonal sequence
$\{ \tilde Q_n(x) \}_{n\geq 0}$ is the specialization of $(1-q)^NZ_N$ at $y=1$. The $N$th moment $\mu_N$ 
of the Al-Salam-Chihara polynomials can now be obtained via the relation:
\[
\mu_N = \sum_{k=0}^N \binom Nk (-1)^{N-k} 2^{-k} (1-q)^k Z_k|_{y=1}.
\]
Actually the methods of Section~\ref{paths} also give a direct proof of the following.

\begin{thm} The $N$th moment of the Al-Salam-Chihara polynomials is:
\begin{equation} \label{asc_mom}
  \begin{split}
  \mu_N = \frac{1}{2^N}
  \sum_{\substack{ 0\leq n \leq N \\ n\equiv N \hbox{\scriptsize mod }2  }}
  \left(  \sum_{j=0}^{\frac{N-n}2}  (-1)^j q^{\binom{j+1}2}  \tqbin{ n+j }{j}
        \left( \binom{N}{\frac{N-n}2-j} - \binom{N}{\frac{N-n}2-j-1}  \right)      
  \right) \\
  \times \left( \sum_{k=0}^n \qbin nk a^k b^{n-k} \right).
  \end{split}
\end{equation}
\end{thm}

\begin{proof} The general idea is to adapt the proof of Theorem~\ref{Z_th} in Section~\ref{paths}.
Let $\mathfrak{P}'_N \subset \mathfrak{P}_N$ be the subset of paths which contain no step 
$\rightarrow$ with weight $1+y$. By Proposition~\ref{mu_ortho}, the sum of weights of elements
in $\mathfrak{P}'_N$ specialized at $y=1$, gives the $N$th moment of the sequence 
$\{Q_n(\frac x2)\}_{n\geq 0}$. This can be seen by comparing the weights in the Motzkin paths 
and the recurrence \eqref{recASC}. But the $N$th moment of this sequence is also $2^N\mu_N$.

From the definition of the bijection $\Phi$ in Section~\ref{paths}, we see that $\Phi(H_1,H_2)$
has no step $\rightarrow$ with weight $1+y$ if and only if $H_1$ has the same property. So
from Proposition~\ref{decomp} the bijection $\Phi^{-1}$ gives a weight-preserving bijection
between $\mathfrak{P}'_N$ and the disjoint union of $\mathfrak{R}'_{N,n} \times \mathfrak{B}^*_n$ 
over $n\in\{0,\dots,N\}$, where $\mathfrak{R}'_{N,n} \subset \mathfrak{R}^*_{N,n}$ is the subset
of paths which contain no horizontal step with weight $1+y$. Note that $\mathfrak{R}'_{N,n}$ 
is empty when $n$ and $N$ don't have the same parity, because now $n$ has to be the number of 
steps $\rightarrow$ in a Motzkin path of length $N$. In particular we can restrict the sum 
over $n$ to the case $n\equiv N$ mod $2$.

At this point it remains only to adapt the proof of Proposition~\ref{RNn} to compute the sum of
weights of elements in $\mathfrak{R}'_{N,n}$, and obtain the sum over $j$ in \eqref{asc_mom}. As 
in the previous case we use Lemma~\ref{decomp2} and Lemma~\ref{core}. But in this case instead of
Motzkin prefixes we get Dyck prefixes, so to conclude we need to know that 
$\tbinom{N}{(N-n)/2-j} - \tbinom{N}{(N-n)/2-j-1}$ is the number of Dyck prefixes of length $N$ and 
final height $n+2i$. The rest of the proof is similar.
\end{proof}

We have to mention that there are analytical methods to obtain the moments 
$\mu_N$ of these polynomials. A nice formula for the Askey-Wilson moments was given by Stanton
\cite{DS}, as a consequence of joint results with Ismail \cite[equation (1.16)]{IS}. As a
particular case they have the Al-Salam-Chihara moments:
\begin{equation} \label{asc_mom2}
  \mu_N = \frac1{2^N} \sum_{k=0}^N (ab;q)_kq^k \sum_{j=0}^k \frac{ q^{-j^2} a^{-2j}
   (q^{j} a + q^{-j}a^{-1})^N}{ (q,a^{-2}q^{-2j+1};q)_{j}(q,a^2q^{1+2j};q)_{k-j}},
\end{equation}
where we use the $q$-Pochhammer symbol.
The latter formula has no apparent symmetry in $a$ and $b$ and has denominators, but Stanton \cite{DS}
gave evidence that \eqref{asc_mom2} can be simplified down to \eqref{asc_mom} using binomial, 
$q$-binomial, and $q$-Vandermonde summation theorems. Moreover \eqref{asc_mom2} is equivalent to 
a formula for rescaled polynomials given in \cite{KSZ08} (Section~4, Theorem~1 and equation (29)).

\subsection{Some particular cases of Al-Salam-Chihara moments}\label{simpli}
When $a=b=0$ in \eqref{asc_mom} we immediately recover the known result for the continuous 
$q$-Hermite moments. This is 0 if $N$ is odd, and the Touchard-Riordan formula if $N$ is 
even. Other interesting cases are the $q$-secant numbers $E_{2n}(q)$ and $q$-tangent numbers
$E_{2n+1}(q)$, defined in \cite{HZR} by continued fraction expansions of the ordinary generating 
functions:
%$\sum_{n\geq0} E_{2n}(q)t^n$ and $\sum_{n\geq0} E_{2n+1}(q)t^n$, which are respectively:
\begin{equation}
\sum_{n\geq0} E_{2n}(q)t^n=
%\tfrac{1}{1-\tfrac{[1]_q^2t}{1-\tfrac{[2]_q^2t}{1-\tfrac{[3]_q^2t}{\ddots}}}}
\cfrac{1}{1-\cfrac{[1]_q^2t}{1-\cfrac{[2]_q^2t}{1-\cfrac{[3]_q^2t}{\ddots}}}}
\displaystyle
\quad\hbox{and}\quad
\sum_{n\geq0} E_{2n+1}(q)t^n=
\scriptstyle
%\tfrac{1}{1-\tfrac{[1]_q[2]_qt}{1-\tfrac{[2]_q[3]_qt}{1-\tfrac{[3]_q[4]_qt}{\ddots}}}}.
\cfrac{1}{1-\cfrac{[1]_q[2]_qt}{1-\cfrac{[2]_q[3]_qt}{1-\cfrac{[3]_q[4]_qt}{\ddots}}}}.
\end{equation}
%These generating functions are respectively the secant and tangent functions when $q=1$.
%WRONG
The exponential generating function of the numbers $E_n(1)$ is the function $\mathrm{tan}(x)
+\mathrm{sec}(x)$. We have the combinatorial interpretation \cite{HZR,MJV2}:
\begin{equation}
E_n(q) = \sum_{\sigma\in\mathfrak{A}_n} q^{\hbox{\scriptsize\mot}(\sigma)},
\end{equation}
where $\mathfrak{A}_n\subset\mathfrak{S}_n$ is the set of alternating permutations, {\it i.e.}
permutations $\sigma$ such that $\sigma(1)>\sigma(2)<\sigma(3)>\dots$. The continued fractions 
show that these numbers are particular cases of Al-Salam-Chihara moments:
\begin{equation}
E_{2n}(q)  =(\tfrac2{1-q})^{2n}\mu_{2n}|_{a=-b=i\sqrt{q}}, \qquad\hbox{and}\quad 
E_{2n+1}(q)=(\tfrac2{1-q})^{2n}\mu_{2n}|_{a=-b=iq}
\end{equation}
(where $i^2=-1$). From \eqref{asc_mom} and a $q$-binomial identity it is possible to obtain
the closed formulas for $E_{2n}(q)$ and $E_{2n+1}(q)$ that were given in \cite{MJV2}, in a similar
manner that \eqref{Z} can be simplified into \eqref{cas1} when $\alpha=\beta=1$. Indeed, from
\eqref{asc_mom} we can rewrite:
\begin{equation}
2^{2n}\mu_{2n} = \sum_{m=0}^n \big(\tbinom{2n}{n-m} - \tbinom{2n}{n-m-1}\big)
\sum_{j,k\geq0} (-1)^j q^{\binom{j+1}2} \tqbin{2m-j}{j}\tqbin{2m-2j}{k} 
 \left(\tfrac ba\right)^k a^{2m-2j}.
\end{equation}
This latter sum over $j$ and $k$ is also
\begin{eqnarray}
   \sum_{j,k\geq0} (-1)^j q^{\binom{j+1}2}
     \tqbin{2m-j}{j+k}\tqbin{j+k}{j} \left(\tfrac ba\right)^k  a^{2m-2j}   
 = \sum_{\ell\geq j\geq0} (-1)^j q^{\binom{j+1}2}
     \tqbin{2m-j}{\ell}\tqbin{\ell}{j} \left(\tfrac ba\right)^{\ell-j}  a^{2m-2j}. 
\end{eqnarray}
The sum over $j$ can be simplified in the case $a=-b=i\sqrt q $, or $a=-b=iq$, using the 
$q$-binomial identities already used in \cite{MJV} (see Lemma~2):
\begin{equation}
\sum_{j\geq0} (-1)^j q^{\binom j 2}\qbin{2m-j}{\ell}\qbin{\ell}{j} = q^{\ell(2m-\ell)},
\end{equation}
and
\begin{equation}
\sum_{j\geq0} (-1)^j q^{\binom {j-1} 2}\qbin{2m-j}{\ell}\qbin{\ell}{j} = 
\tfrac{  q^{(\ell+1)(2m-\ell)} - q^{\ell(2m-\ell)} +q^{\ell(2m-\ell+1)} - 
  q^{(\ell+1)(2m-\ell+1)} }{ q^{2m-1}(1-q)}.
\end{equation}
Omitting details, this gives a new proof of the Touchard-Riordan-like formulas \cite{MJV2}:
\begin{equation} \label{En}
E_{2n}(q) = \frac 1{(1-q)^{2n}} \sum_{m=0}^n \left( \tbinom{2n}{n-m}-
\tbinom{2n}{n-m-1} \right)
\sum_{\ell=0}^{2m} (-1)^{\ell+m} q^{\ell(2m-\ell)+m}
\end{equation}
and
\begin{equation} \label{En2}
E_{2n+1}(q) = \frac 1{(1-q)^{2n+1}} \sum_{m=0}^n \left( \tbinom{2n+1}{n-m}-
\tbinom{2n+1}{n-m-1} \right) \sum_{\ell=0}^{2m+1} (-1)^{\ell+m} q^{\ell(2m+2-\ell)}.
\end{equation}

%%%%%%%%%%%%%%%%%%%%%%%%%%%%%%%%%%%%%%%%%%%%%%%%%%%%%%%%%%%%%%%%
\section{Some classical integer sequences related to $\bar Z_N$}
%%%%%%%%%%%%%%%%%%%%%%%%%%%%%%%%%%%%%%%%%%%%%%%%%%%%%%%%%%%%%%%%
\label{num}

It should be clear from the interpretation given in \eqref{main_int} that the polynomial 
$\bar Z_N$ contains quite a lot a of combinatorial information. When $\alpha=\beta=1$, 
the coefficients of $y^k$ in $\bar Z_n$ are the $q$-Eulerian numbers introduced by 
Williams \cite{LW}:
\begin{equation} \label{ZEhat}
  \bar Z_N  |_{\alpha=\beta=1}  = \sum_{k=0}^N y^k \hat E_{k+1,N+1}(q),
\end{equation}
where $\hat E_{k,n}(q)$ is defined in \cite[Section~6]{LW}. It was proved by Williams, that
$\hat E_{k,n}(q)$ is equal to the Eulerian number $A_{n,k}$ when $q=1$, to the binomial
coefficient $\binom{n-1}{k-1}$ when $q=-1$, and to the Narayana number 
$N_{n,k}=\frac1n\binom nk \binom n{k-1}$ when $q=0$. With the other parameters $\alpha$ and $\beta$,
there are other interesting results.

\subsection{Stirling numbers}

Carlitz $q$-analog of the Stirling numbers of the second kind, denoted by $S_2[n,k]$,
are defined when $1\leq k \leq n$ by the recurrence \cite{Car48}:
\begin{equation} \label{rec_stir}
    S_2[n,k]=S_2[n-1,k-1] + [k]_q S_2[n-1,k], \quad S_2[n,k]=1 \hbox{ if } k=1 \hbox{ or } k=n.
\end{equation}

\begin{prop}
If $\alpha=1$, the coefficient of $\beta^ky^{k}$ in $\bar Z_N$ is $S_2[N+1,k+1]$.
\end{prop}

\begin{proof}
This follows from the interpretation \eqref{ZTP} in terms of permutation tableaux (see Definition
\ref{def_PT}). Indeed, the coefficient of $\beta^ky^k$ in $\bar Z_N$ counts permutation tableaux 
of size $N+1$, with $k+1$ rows, and $k+1$ unrestricted rows. In a permutation tableau with no 
restricted row, each column contains a sequence of 0's followed by a sequence of 1's. Such permutation
tableaux follow the recurrence \eqref{rec_stir} where $n$ is the size and $k$ is the number of rows. 
Indeed, if the bottom row is of size 0 we can remove it and this gives the term $S_2[n-1,k-1]$. 
Otherwise the first column is of size $k$, this gives the term $[k]_q S_2[n-1,k]$ because the factor 
$[k]_q$ accounts for the possibilities of the first column, the factor $S_2[n-1,k]$ accounts for what 
remains after removing the first column.
\end{proof}

The proof only relies on simple facts about permutation tableaux, and with no doubts it was 
previously noticed that $S_2[n,k]$ appears when we count permutation tableaux without restricted 
rows. Actually permutation tableaux with no restricted rows are equivalent to the 0-1 tableaux 
introduced by Leroux \cite{Ler90} as a combinatorial interpretation of $S_2[n,k]$.

From \eqref{Z}, it is possible to obtain a formula for $S_2[n,k]$. First, observe that the 
coefficient of $y^k$ in $\bar Z_N$ has degree $k$ in $\beta$. Hence, from the previous proposition:
\begin{equation} \label{stirZ}
  \sum_{k=0}^{N} a^k S_2[N+1,k+1] = \lim_{y\to 0} \bar Z_N(1,\tfrac a y,y,q).
\end{equation}
We have $R_{N,n}(0,q)=\binom Nn$. When $\alpha=1$ and $\beta=\frac ya$, we have $\alp=-q$ and 
$y\bet=(1-q)a+y$. So from \eqref{Z} and \eqref{stirZ} it is straightforward to obtain:
\begin{equation} \label{carl1}
  S_2[N+1,k+1] = \frac 1{(1-q)^{N-k}} \sum_{j=0}^{N-k} (-q)^j \binom{N}{k+j}\qbin{k+j}{j}.
\end{equation}
Note that this differs from the expression originally given by Carlitz \cite{Car48}:
\begin{equation} \label{carl2}
  S_2[N,k] = \frac 1{(1-q)^{N-k}} \sum_{j=0}^{N-k} (-1)^j \binom{N}{k+j}\qbin{k+j}{j},
\end{equation}
but it is elementary to check that \eqref{carl1} and \eqref{carl2} are equivalent, 
using the two-term recurrence relations for binomial and $q$-binomial coefficients.

When $y=\alpha=1$, the coefficient of $\beta^k$ in $\bar Z_N$ is a $q$-analog of the Stirling 
number of the first kind $S_1(N+1,k+1)$. It is such that $q$ counts the number of patterns 31-2 
in permutations of size $N+1$ and with $k+1$ right-to-left minima. Knowing the symmetry \eqref{Z_sym},
we could also say that it is such that $q$ counts the number of patterns 31-2 
in permutations of size $N+1$ and with $k+1$ right-to-left maxima.
%we have the equality:
%\begin{equation} \label{stir1sym}
%  \sum_{\substack{ \sigma\in\mathfrak{S}_{N+1} \hbox{\scriptsize \;having} \\ k + 1 
%   \hbox{\scriptsize \; right-to-left minima }}}  q^{\hbox{\scriptsize \mot}(\sigma)}
%= \sum_{\substack{ \sigma\in\mathfrak{S}_{N+1} \hbox{\scriptsize \;having} \\ k + 1
%   \hbox{\scriptsize \; right-to-left maxima }}}  q^{\hbox{\scriptsize \mot}(\sigma)}.
%\end{equation}
The combinatorial way to see the symmetry \eqref{Z_sym} is the transposition of permutation tableaux
\cite{CW3}, so at the moment it is quite indirect to see that the two interpretations agree since we need 
all the bijections from Section~\ref{bij}. We have no knowledge of previous work concerning these
$q$-Stirling numbers of the first kind.

\subsection{Fine numbers}

The sequence of Fine numbers shares many properties with the Catalan numbers, we refer to
\cite{DSh} for history and facts about them. We will show that a natural symmetric refinement 
of them appears as a specialization of $\bar Z_N$.

A {\it peak} of a Dyck path is a factor 
$\nearrow\searrow$, we denote by $\pk(P)$ the number of peaks of a path $P$. A Fine path is a Dyck path 
$D$ such that there is no factorization $D=D_1\nearrow\searrow D_2$ where $D_1$ and $D_2$ are Dyck 
paths. The Fine number $F_n$ is the number of Fine path of length $2n$, and more generally the 
polynomial $F_n(y)$ is $\sum y^{\pk(P)}$ where the sum is over Fine paths $P$ of length $2n$.
These polynomials were considered in \cite{DSh} via their generating function.

An interesting property is that $F_n(y)$ is self-reciprocal, {\it i.e.}  
$F_n(y) = y^{n}F_n(\frac 1y)$ (a simple proof of this will appear below). This is 
reminiscent of the Dyck paths: the number of Dyck paths of length $2n$ with $k$ peaks is the 
Narayana number $N_{k,n}$ and we have $N_{k,n}=N_{n+1-k,n}$. The first values are:
\begin{equation}\begin{split}
  F_1(y) = 0, \quad F_2(y) = y, \quad F_3(y) =y^2+y , \quad F_4(y) = y^3+4y^2+y, 
\\
   F_5(y) = y^4+8y^3+8y^2+y, \quad F_6(y) = y^5+13y^4+29y^3+13y^2+y.
\end{split}\end{equation}

\begin{prop} 
When $\frac1\alpha=-y$, $q=0$, and $\beta=1$, we have $\bar Z_N= F_N(y)$.
\end{prop}

\begin{proof}
In this case we have $\bet=0$, $\alp=-1-y$. From the weights in the general case \eqref{Z_w}, 
we see that now $Z_N$ is the sum of weights of Motzkin paths such that:
\begin{itemize}
\item the weight of a step $\nearrow$ is 1, the weight of a step $\searrow$ is $y$,
\item the weight of a step $\rightarrow$ is $1+y$, but there is no such step at height 0.
\end{itemize}
Let $H(t,y)=\sum_{N\geq 0} Z_Nt^N$. It is such that $H(t,y) = 1+yt^2G(t,y)H(t,y)$, where
$G(t,y)$ counts the paths with the same weights but possibly with steps $\rightarrow$ at height
0. Let $L(t,y)=\sum t^{\ell(P)}y^{\pk(P)}$ where $\ell(P)$ is half the number of steps of $P$, 
and the sum is over all Dyck paths $P$. Some standards arguments show that $G(t,y)$ is linked
with Narayana numbers in such a way that $L(t,y)=1+ytG(t,y)$. 
So we have $H(t,y) = 1+t\big(L(t,y)-1\big)H(t,y)$, which is precisely the functional equation given 
in \cite[Section~7]{DSh} for the generating function $\sum F_n(y)t^n$. This completes the proof.
\end{proof}

\smallskip

When we substitute $y$ with $\frac1y$ in the Motzkin paths considered in the proof, we see that 
the weight of a step $\rightarrow$ is divided by $y$ and the weight of a step $\searrow$ is 
divided by $y^2$, so the total weight is divided by $y^n$ where $n$ is the length of the path. 
This proves the symmetry of the coefficients of $F_n(y)$.

\smallskip

Note that the symmetry of $Z_N$ obtained in this section is not a particular case of previously
known symmetry \eqref{Z_sym}. It may be a special case of another more general symmetry.

%%%%%%%%%%%%%%%%%%%%%%%%%%%%
\section{Concluding remarks}
%%%%%%%%%%%%%%%%%%%%%%%%%%%%

We have used two kinds of weighted Motzkin paths to study $Z_N$. The first kind are the 
elements of $\mathfrak{P}_N$, {\it i.e.} the paths coming from the matrices $D$ and $E$ 
defined in \eqref{defD} and \eqref{defE}. They have the property that the weight of a
step only depends on its direction and its height, so that there is a J-fraction expansion 
for the generating function $\sum_{N\geq0}Z_Nt^N$ with the four parameters $\alpha$, $\beta$, 
$y$ and $q$. The second kind of weighted Motzkin paths are the Laguerre histories, and their
nice property is that they are linked bijectively with permutations. One might ask if there is 
a set of weighted Motzkin paths having both properties, but its existence is doubtful. Still it 
could be nice to have a direct simple proof that these two kinds of paths give the same quantity
$Z_N$.

Our two new combinatorial interpretations in Theorems \ref{histoires} and \ref{main} 
complete the known combinatorial interpretations \eqref{ZTP} and \eqref{depart}, and this
makes at least four of them. Although all is proved bijectively, there is not a direct 
bijection for any pair of combinatorial interpretations. In particular it would be nice to 
have a more direct bijection between permutation tableaux and permutations to link the 
right-hand sides of \eqref{ZTP} and \eqref{main_int}, instead of composing four bijections
(Steingr\'imsson-Williams, reverse complement of inverse, Foata-Zeilberger and Françon-Viennot).
Permutation tableaux are mainly interesting because of their link with permutations, so in 
this regard it is desirable to have a direct bijection preserving the four parameters considered
here.

We have given evidence that the lattice paths are good combinatorial objects to study the
PASEP with three parameters. However, our combinatorial interpretation of $Z_N$ with the
Laguerre histories relies on the previous one with permutation tableaux. To complete the lattice 
paths approach, it might be interesting to have a direct derivation of stationary probabilities 
in terms of Laguerre histories. For example in \cite{CW1}, Corteel and Williams define a Markov 
chain on permutation tableaux which projects to the PASEP, similarly we could hope that there is 
an explicit simple description of such a Markov chain on Laguerre histories.

The three-parameter PASEP is now quite well understood since we have exact expressions for 
many interesting quantities. In a more general model, we allow particles to enter the 
rightmost site, and exit the leftmost site, so that there are five parameters. In this 
case the partition function is linked with the Askey-Wilson moments, in a similar manner 
that the three-parameter partition function is linked with Al-Salam-Chihara moments \cite{USW}. 
Recently, Corteel and Williams \cite{CW3} showed that there exist some staircase tableaux 
generalizing permutation tableaux, arising from this general model with five parameters. It 
is not clear whether a closed formula for the five-parameter partition function exists, and 
in the case it exists it might be unreasonably long. But knowing the results about the 
three-parameter partition function, we expect the five-parameter partition function to be 
quite full of combinatorial meaning.

%%%%%%%%%%%%%%%%%%%%%%%%%%%%


\begin{thebibliography}{999}
%%%%%%%%%%%%%%%%%%%%%%%%%%%%

\bibitem{ASCh}
 W. A. Al-Salam and T. S. Chihara, Convolutions of orthonormal polynomials, 
 SIAM J. Math. Anal. 7 (1976), 16--28. 

\bibitem{ASCa}
 W. A. Al-Salam and L. Carlitz, Some orthogonal $q$-polynomials, 
 Math. Nachr. 30 (1965), 47--61. 

\bibitem{BHPSD}
 P. Blasiak, A. Horzela, K. A. Penson, A. I. Solomon, and G. H. E. Duchamp,
 Combinatorics and Boson normal ordering: A gentle introduction,
 Am. J. Phys. 75 (2007), 639--646.

\bibitem{BE}
 R. A. Blythe and M. R. Evans, Nonequilibrium steady states of matrix product form: A solver's guide, 
 J. Phys. A: Math. Gen. 40 (2007), 333--441.

\bibitem{BECE}
 R. A. Blythe, M. R. Evans, F. Colaiori and F. H. L. Essler,
 Exact solution of a partially asymmetric exclusion model using a deformed
 oscillator algebra, J. Phys. A: Math. Gen. 33 (2000), 2313--2332.

\bibitem{BCEPR} 
 R. Brak, S. Corteel, J. Essam, R. Parviainen and A. Rechnitzer,
 A combinatorial derivation of the PASEP stationary state,
 Electron. J. Combin.  13(1)  (2006), R108.

\bibitem{BGR}
 R. Brak, J. de Gier and V. Rittenberg, Nonequilibrium stationary states and equilibrium
 models with long range interactions, J. Phys. A: Math. Gen. 37 (2004), 4303--4320.

\bibitem{Car48}
 L. Carlitz, $q$-Bernoulli numbers and polynomials, Duke Math. J. 15 (1948), 987--1000.

\bibitem{Co}
 S. Corteel, Crossings and alignments of permutations, Adv. in Appl. Math.
 38(2) (2007), 149--163.

\bibitem{CN}
 S. Corteel, P. Nadeau, Bijections for permutation tableaux, European. J. Combin. 30(1)
 (2009), 295--310.

\bibitem{CW1}
 S. Corteel and L. K. Williams, A Markov chain on permutations which projects to the PASEP,
 Int. Math. Res. Not. (2007), article ID rnm055.

\bibitem{CW2}
 S. Corteel and L. K. Williams, Tableaux combinatorics for the asymmetric exclusion process,
 Adv. in Appl. Math. 39(3) (2007), 293--310.

\bibitem{CW3}
 S. Corteel and L. K. Williams, Tableaux combinatorics for the asymmetric exclusion process 
 and Askey-Wilson polynomials, arXiv:0910.1858 [math.CO].

\bibitem{CJPR}
 S. Corteel, M. Josuat-Vergès, T. Prellberg and M. Rubey,
 Matrix Ansatz, lattice paths and rook placements, Proc. FPSAC 2009.

\bibitem{CU}
 A. Cuyt, A. B. Petersen, B. Verdonk, H. Waadeland, W. B. Jones, Handbook of continued 
 fractions for special functions, Springer, 2008. 

\bibitem{DEHP}
 B. Derrida, M. Evans, V. Hakim and V. Pasquier, Exact solution of a 1D asymmetric
 exclusion model using a matrix formulation, J. Phys. A: Math. Gen. 26 (1993),
 1493--1517.

\bibitem{DSh}
 E. Deutsch and L.W. Shapiro, A survey of the Fine numbers, Discrete Math. 241 (2001), 241--265.

\bibitem{Fla}
 P. Flajolet, Combinatorial aspects of continued fractions, Discrete Math. 41 (1982), 145--153.

\bibitem{FZ}
 D. Foata and D. Zeilberger, Denert's permutation statistic is indeed Euler-Mahonian, Stud.
 Appl. Math. 83(1) (1990), 31--59.

\bibitem{FV}
 J. Françon and X. G. Viennot, Permutations selon leurs pics, creux, doubles 
 montées et double descentes, nombres d'Euler et nombres de Genocchi, Discrete 
 Math. 28(1) (1979), 21--35.

\bibitem{HZR}
 G.-N. Han, A. Randrianarivony, J. Zeng, Un autre $q$-analogue des nombres 
 d'Euler, Séminaire Lotharingien de Combinatoire 42 (1999), Article B42e.

\bibitem{IS}
 M. Ismail and D. Stanton, $q$-Taylor theorems, polynomial expansions, and
 interpolation of entire functions, J. Approx. Th. 123 (2003), 125--146.

\bibitem{MJV}
 M. Josuat-Vergès, Rook placements in Young diagrams and permutation enumeration, 
 Adv in Appl. Math. (doi:10.1016/j.aam.2010.04.003).

\bibitem{MJV2}
 M. Josuat-Vergès, A $q$-enumeration of alternating permutations, 
 European J. Combin. 31 (2010), 1892--1906.

\bibitem{KSZ08}
 A. Kasraoui, D. Stanton and J. Zeng, The Combinatorics of Al-Salam-Chihara $q$-Laguerre polynomials,
 Adv. in Appl. Math. (doi:10.1016/j.aam.2010.04.008).

\bibitem{DK}
 D. Kim, On combinatorics of Al-Salam Carlitz polynomials, 
 European. J. Combin. 18(3) (1997), 295--302.

\bibitem{KoSw98}
 R. Koekoek and R. F. Swarttouw, The Askey-scheme of hypergeometric orthogonal 
 polynomials and its $q$-analogue, Delft University of Technology, 
 Report no. 98--17 (1998).

\bibitem{Ler90}
 P. Leroux, Reduced matrices and $q$-log concavity properties of $q$-Stirling numbers, 
 J. of Comb. Theory A 54 (1990), 64--84.

\bibitem{MV}
 A. de Médicis and X. G. Viennot, Moments des $q$-polynômes de Laguerre et la bijection de 
 Foata-Zeilberger, Adv. in Appl. Math. 15 (1994), 262--304. 

\bibitem{TS}
 T. Sasamoto, One-dimensional partially asymmetric simple exclusion process with open 
 boundaries: orthogonal polynomials approach, J. Phys. A: Math. Gen. 32 (1999), 7109--7131.

\bibitem{DS}
 D. Stanton, personal communication.

\bibitem{SW} 
 E. Steingrímsson and L. K. Williams, Permutation tableaux and
 permutation patterns, J. Combin. Theory Ser. A 114(2) (2007), 211--234.

\bibitem{XGV}
 X. G. Viennot, Une théorie combinatoire des polynômes 
 orthogonaux, Notes de cours, UQAM, Montréal 1988. \\
\url{http://web.mac.com/xgviennot/Xavier_Viennot/livres.html}

\bibitem{USW}
 M. Uchiyama, T. Sasamoto et M. Wadati, Asymmetric simple exclusion process with open 
 boundaries and Askey-Wilson polynomials, J. Phys. A: Math. Gen. 37 (2004), 4985--5002.

\bibitem{AV}
 A. Varvak, Rook numbers and the normal ordering problem,
 J. Combin. Theory Ser. A 112(2) (2005), 292--307.

\bibitem{LW}
 L. K. Williams,  Enumeration of totally positive Grassmann cells,
 Adv. Math. 190(2) (2005), 319-342.

\end{thebibliography}
\end{document}